\documentclass{article}
\usepackage[utf8]{inputenc}
\usepackage{amsmath}
\usepackage{amsthm}
\usepackage{amsfonts}
\usepackage{amssymb, amsbsy}
\usepackage{color}
\usepackage{hyperref}
\usepackage{enumitem}
\usepackage{comment}

\usepackage[a4paper, textwidth=14cm, top=3cm]{geometry}

\usepackage{mathtools}
\usepackage{cancel}

\title{Derivation of a Biot-Plate-System for a thin poroelastic layer}
\author{Markus Gahn}

\date{}

\newcommand{\R}{\mathbb{R}}
\newcommand{\N}{\mathbb{N}}
\newcommand{\Z}{\mathbb{Z}}

\newcommand{\veps}{v_{\varepsilon}}
\newcommand{\ueps}{u_{\varepsilon}}
\newcommand{\peps}{p_{\varepsilon}}
\newcommand{\weps}{w_{\varepsilon}}
\newcommand{\tueps}{\widetilde{u}_{\varepsilon}}
\newcommand{\Veps}{V_{\varepsilon}}
\newcommand{\Ueps}{U_{\varepsilon}}
\newcommand{\Weps}{W_{\varepsilon}}
\newcommand{\tUeps}{\widetilde{U}_{\varepsilon}}
\newcommand{\phieps}{\phi_{\varepsilon}}
\newcommand{\vareps}{\varepsilon}
\newcommand{\oeps}{\Omega_{\varepsilon}}
\newcommand{\oef}{\Omega_{\varepsilon}^f}
\newcommand{\oes}{\Omega_{\varepsilon}^s}
\newcommand{\geps}{\Gamma_{\varepsilon}}
\newcommand{\sepm}{S_{\varepsilon}^{\pm}}

\newcommand{\tphieps}{\widetilde{\phi}_{\varepsilon}}
\newcommand{\foe}{\frac{1}{\varepsilon}}
\newcommand{\bxfxe}{\left(\bar{x},\frac{x}{\varepsilon}\right)}
\newcommand{\btxfxe}{\left(t,\bar{x},\frac{x}{\varepsilon}\right)}
\newcommand{\x}{\bar{x}}
\newcommand{\rats}{\overset{2}{\rightharpoonup}}

\renewcommand{\u}{\overline{u}}
\renewcommand{\v}{\overline{v}}
\newcommand{\w}{\overline{w}}
\newcommand{\z}{\overline{z}}
\newcommand{\spaceV}{\mathcal{V}}
\newcommand{\spaceW}{\mathcal{W}}
\renewcommand{\div}{\mathrm{div}}

\newcommand{\tbxfxe}{\left(t,\bar{x},\frac{x}{\varepsilon}\right)}
\newcommand{\app}{\mathrm{app}}
\newcommand{\fxe}{\frac{x}{\varepsilon}}

\newcommand{\ie}{i.\,e., }

\newtheorem{definition}{Definition}
\newtheorem{remark}{Remark}
\newtheorem{theorem}{Theorem}
\newtheorem{proposition}{Proposition}
\newtheorem{lemma}{Lemma}
\newtheorem{corollary}{Corollary}

\begin{document}

\maketitle
\vspace{-2em}
\begin{center}
\begin{minipage}{25em}
\centering\textit{Institute for Mathematics, University Heidelberg, \\
Im Neuenheimer Feld 205, 69120 Heidelberg, Germany.}
\end{minipage}
\end{center}

\vspace{2em}
\begin{abstract}
We study incompressible fluid flow through a thin poroelastic layer and rigorously derive a macroscopic model when the thickness of the layer tends to zero. Within the layer we assume a periodic structure and both, the periodicity and the thickness of the layer, are of order $\varepsilon$ which is small compared to the length of the layer. The fluid flow  is described by  quasistatic Stokes-equations and for the elastic solid we consider linear elasticity equations, and both are coupled via continuity of the velocities and the normal stresses. The aim is to pass to the limit $\varepsilon \to 0$ in the weak microscopic formulation by using multi-scale techniques adapted to the simultaneous homogenization and dimension reduction in continuum mechanics. The macroscopic limit model is given by a coupled Biot-Plate-system consisting of a generalized Darcy-law coupled to a Kirchhoff-Love-type plate equation including the Darcy pressure.
\end{abstract}

\section{Introduction}

In this paper we rigorously derive a Biot-Plate system for a linearized quasistatic fluid-structure interaction problem in a thin periodic poroelastic layer. The thin layer consists of a fluid and an elastic solid part, and the periodicity of the structure within the layer and the thickness of the layer are both of order $\varepsilon$, where the paramter $\varepsilon$ is small compared to the length of the layer. We assume small deformations, such that in the solid part the equations of linear elasticity hold. The fluid is supposed to be incompressible and its evolution is given by Stokes-equations. Between the fluid and the solid we assume a linearized (fluid equation also formulated in Lagrangian-coordinates) coupling condition, where we assume continuity of the velocities and normal stresses. 
We consider the case when the ratio between the viscosity and the characteristic size of the elasticity coefficients is of order $\vareps^2$ corresponding to the diphasic cases (we refer to \cite{ClopeauFerrinGilbertc2001,GilbertMikelic2000} for more details on this contrast of property number). In the limit $\vareps \to 0$ this scaling leads to a Biot-type system.

To pass to the limit $\vareps \to 0$ we have to deal simultaneously with the homogenization of the porous structure in the layer and the reduction of the thin layer to a lower dimensional manifold $\Sigma$. Hence, multi-scale techniques combining homogenization and dimension reduction are necessary. Here we use the method of two-scale convergence in thin periodic structures \cite{NeussJaeger_EffectiveTransmission}. In a first step we establish existence and uniqueness of a microscopic solution in the thin porous layer and prove some additional regularity results with respect to time which are later necessary for the uniform a priori estimates. For the existence proof we use the Galerkin-method. Due to the quasistatic character of the problem the underlying system of ordinary differential equations for the Galerkin coefficients has no full rank and we have to argue in a slightly different way compared to standard literature on parabolic problems. The basic ingredients for the derivation of the macroscopic problem are uniform a priori estimates with respect to the scaling parameter $\vareps$. To estimate the fluid velocity $\veps$ and the displacement $\ueps$, we have to use a Korn inequality for thin (perforated) domains with constants depending explicitly on $\vareps$, see \cite{GahnJaegerTwoScaleTools}. However, the estimates obtained in this way are not good enough to pass to the limit $\vareps \to 0$. Hence, to obtain better estimates with respect to $\vareps$, it is suitable to consider the quantity $\veps  - \partial_t \ueps$, which has zero boundary values on the fluid-structure interface. However, for this we have to extend the displacement into the fluid domain. Here we use an extension operator from \cite{GahnJaegerTwoScaleTools} which allows to control the symmetric gradient of the extension. Due to the continuity of the velocities at the fluid-structure interface we also need a priori estimates for the time derivative $\partial_t \ueps$ and its (symmetric) gradient. For this we have to take into account the higher regularity with respect to time obtained in the existence result for the microscopic solution.
 To estimate the fluid pressure $\peps$ we prove the existence of a Bogovskii-operator in thin perforated domains and give an explicit bound for its operator norm.
 
Based on the a priori estimates  we are able to derive compactness results for the microscopic solution. As an underlying topology we use the two-scale convergence in thin heterogeneous domains. Because of the thin structure we obtain a specific scaling with respect to $\vareps$ in the a priori estimates for the displacement. As a result, the two-scale limit of the displacement is a Kirchhoff-Love displacement, see for example \cite{gahn2022derivation} and \cite{griso2020homogenization}. In our case, the crucial point is to find the relation between the Kirchhoff-Love displacement and the limit of the fluid velocity. We emphasize that for our compactness results we do not explicitly use that $(\veps,\peps,\ueps)$ is the solution of the microscopic problem, but only the a priori estimates obtained before. Hence, our results can be transferred to more general problems including for example transport of species and also other geometries like bulk domains separated by thin membranes. Finally, using the compactness results for the micro-solutions we derive the macroscopic model. It is well-known, see for example \cite{ClopeauFerrinGilbertc2001}, that for porous domains the macro-model is given by a Biot-system consisting of a generalized Darcy-law for the limit pressure and an homogenized elasticity equation coupled to the limit pressure. In our case, for the thin poroelastic layer, the generalized Darcy-law has to be adapted to the Kirchhoff-Love structure of the limit displacement. For its derivation we first obtain a two-pressure Stokes model on $\Sigma$, which is obtained by a  two-pressure decomposition adapted to the lateral boundary conditions considered in our microscopic model. The limit problem for the displacement is given by a (fourth-order) plate equation, also coupled to the limit pressure. 

There is a large number of results dealing with the heuristic or rigorous derivation of Biot-laws from first principles in poroelastic domains (not thin), and we mention some important contributions. First formal results where obtained in \cite{levy1979propagation,mei1989mechanics}. A first rigorous result can be found in \cite{nguetseng1990asymptotic} for slightly compressible fluids using the method of two-scale convergence. The first rigorous result for an incompressible fluid was obtained in \cite{ClopeauFerrinGilbertc2001}. In \cite{JaegerMikelicNeussRaduHomogenization} an additional reactive flow was taken into account. For the dimension reduction of thin homogeneous structures in linear elasticity we refer to the monograph \cite{ciarlet1997mathematical}. First results for heterogeneous structures were obtained in \cite{caillerie1984thin}, where the heterogeneity occurs in oscillating coefficients and not in the geometry. The problem of perforated thin layers was first treated in \cite{griso2020homogenization} using the unfolding method. A general framework for multi-scale techniques for problems in thin perforated domains was developed in \cite{GahnJaegerTwoScaleTools}. In \cite{marciniak2015rigorous} the rigorous dimension reduction for a Biot-model in a thin layer is performed. More precisely it is assumed that the layer has a porous structure with porosity small compared to the thickness of the layer. In this case, it is suitable to assume that the Biot-law holds in the thin layer. Hence, the limit is not obtained by a simultaneous homogenization and dimension reduction. Results combining these problems for (linear) fluid-structure interaction seem to be rare. The only rigorous result in this direction can be found in \cite{gahn2022derivation}, where also a coupling of the thin layer to bulk domains is considered. However, the underlying scaling of the equations in the thin layer is different from our scaling and does not lead to a Biot-law in the limit. In summary, compared to the existing literature, the main new contribution of this paper is to take into account the simultaneous homogenization and dimension reduction for incompressible fluid flow through poroelastic (thin) layers in the diphasic case leading to a Biot-type macroscopic equation including a plate equation for the macroscopic displacement.

The paper is organized as follows: In Section \ref{sec:micro_model} we introduce the microscopic model with its weak formulation. In Section \ref{sec:main_results} we formulate the main results of the paper and give the macroscopic model. Existence and uniqueness for the micro-model is shown in Section \ref{sec:Existence}. In Section \ref{sec:Apriori_Estimates} we prove the a priori estimates for the microscopic solution and in Section \ref{sec:compactness_results} we give the compactness results. Finally, in Section \ref{sec:derivation_macro_model} we derive the macroscopic model. In the appendix we formulate some technical results. In Appendix \ref{sec:two_pressure_decomp} we deal with function spaces suitable for the derivation of the two-pressure Stokes-model which builds the basis for the (generalized) Darcy-law. In Appendix \ref{sec:Bogovskii} we construct a Bogovskii-operator for thin perforated layers. Some important inequalities known from existing literature are summarized in Appendix \ref{sec:inequalities} and in Appendix \ref{sec:two_scale_convergence} we give a brief introduction to the two-scale convergence in thin perforated domains.

\subsection{Notations}

In the following, let $\Omega \subset \R^n$ be open with Lipschitz-boundary. We usually denote the outer unit normal on $\partial \Omega$ by $\nu$. Further, for $d\in \N$ we denote by $L^p(\Omega)^d$ and $W^{1,p}(\Omega)^d$ the standard Lebesgue and Sobolev spaces with $p \in [1,\infty]$. For $p=2$ we write $H^1(\Omega)^d:= W^{1,2}(\Omega)^d$. For the norms we neglect the upper index $d$, for example we write $\|\cdot\|_{L^p(\Omega)}$ instead of $\|\cdot \|_{L^p(\Omega)^d}$. For $\Gamma \subset 
 \partial \Omega$, we define 
\begin{align*}
H^1(\Omega,\Gamma):= \left\{ u \in H^1(\Omega) \, : \, u = 0 \,\mbox{ on } \Gamma \right\}.
\end{align*}
For a Banach space $X$ we denote its dual by $X'$, and for $p \in [1,\infty]$ we denote the usual Bochner spaces by $L^p(\Omega,X)$. For $\Omega = (0,T)$ with $T>0$ and $k\in \N$ we denote by $W^{k,p}((0,T),X)$ the Bochner spaces of functions $u \in L^p((0,T),X)$ with generalized time-derivatives $\partial_t^l u \in L^p((0,T),X)$ for $l=0,\ldots,k$. Again, for $p=2$ we write $H^k((0,T),X):= W^{k,2}((0,T),X)$.

Let $Z:= Y \times (-1,1) := (0,1)^{n-1} \times (-1,1)$. With the index $\#$ we usually indicate function spaces of $Y$-periodic functions. More precisely, $C_{\#}^{\infty}(Z)$ is the space of smooth $Y$-periodic functions, \ie $f \in C_{\#}^{\infty}(Z)$ fulfills $f \in C^{\infty}(\R^{n-1} \times (-1,1))$ and $f(y + e_i) = f(y)$ for all $y \in \R^{n-1}\times (-1,1)$ and $i=1,\ldots,n-1$. $H^k_{\#}(Z)$ for $k\in \N_0$ is the closure of $C_{\#}^{\infty}(Z)$ in $H^k(Z)$. For a subset $Z_{\ast} \subset Z$ we denote by $H_{\#}^k(Z_{\ast})$ the restriction of functions from $H_{\#}^k(Z)$ to $Z_{\ast}$. Further, we define for $\Gamma:= \partial Z_{\ast} \setminus \partial Z$
\begin{align}\label{def:Hper_Gamma}
H_{\#}^1(Z_{\ast},\Gamma):= \{f \in H_{\#}^1(Z_{\ast})\, : \, f = 0 \mbox{ on }\Gamma\}.
\end{align}
For a function $f:\Omega \rightarrow \R^m$ with $m\in \N$ we denote its (weak) Jacobi matrix by $\partial f$ with $(\partial f)_{ij} = \partial_j f_i$ and the (weak) gradient $\nabla f:= (\partial f)^t$. For a vector field $u:\Omega \rightarrow \R^n$ we denote the symmetric gradient by 
\begin{align*}
D(u):= \frac12 \left(\nabla u + \nabla u^t\right).
\end{align*}
For $x \in \R^n$ we write $\x := (x_1,\ldots,x_{n-1})$ so we have $x= (\x,x_n)$. Let $\Sigma \subset \R^{n-1}$ open ($n\geq 2$). Then we denote the gradient of a function $f:\Sigma \rightarrow \R$ by $\nabla_{\x} f$. We also consider $\nabla_{\x} f(\x)$ as a vector in $\R^n$ by setting the last component equal to $0$. For a vector field $u:\Sigma \rightarrow \R^{n-1}$ we denote its symmetric gradient by $D_{\x}(u)$ and also consider if necessary $D_{\x}(u)(\x)$ as an element in $\R^{n\times n}$ by setting $D_{\x}(u)(\x)_{ij}= 0$ for $i=n$ or $j=n$. For  $v:\Sigma \rightarrow \R^n$ we use the same notation for the symmetric gradient:
\begin{align*}
D_{\x}(v)_{ij} := \frac12(\partial_i v_j + \partial_j v_i).
\end{align*}
Hence $\partial_n v_j = 0$ for $j=1,\ldots,n$.
For matrices $A,B \in \R^{n\times n}$ we define the Frobenius inner product by  $A:B = \mathrm{tr}(AB^t)$. The unit matrix is denoted by $I$.
Finally, throughout the paper $C>0$ denotes a generic constant which is independent of $\vareps$.

\section{The microscopic model}
\label{sec:micro_model}

In this section we formulate the micro-model, starting with a detailed description of the underlying geometry of the model:

\subsection{Microscopic geometry}

Let $n\in \N$ with $n\geq 2$ and $\Sigma = (a,b) \in \R^{n-1}$ with $a,b \in \mathbb{Q}^{n-1}$ and $a_i < b_i$ for $i=1,\ldots,n-1$. Further, we assume that $\vareps>0$ with $\vareps^{-1}\in \N$.
We define the whole layer by 
\begin{align*}
\oeps:= \Sigma \times (-\vareps , \vareps).
\end{align*}
Within the  thin layer we have a fluid part $\oef$ and a solid part $\oes$, which have a periodical microscopic structure. More precisely, we define  the reference cell
\begin{align*}
Z := Y\times (-1,1) := (0,1)^{n-1} \times (-1,1),
\end{align*}
with top and bottom
\begin{align*}
S^{\pm} := Y \times \{\pm 1\}, \qquad S:= S^+ \cup S^-.
\end{align*}
The cell $Z$ consists of a solid part $Z_s\subset Z $ and a fluid part $Z_f \subset Z$ with common interface $\Gamma = \mathrm{int}\left( \overline{Z_s} \cap \overline{Z_f} \right)$ (here \textit{int} denotes the interior of a set). Hence, we have 
\begin{align*}
Z = Z_f \cup Z_s \cup \Gamma.
\end{align*}
If it is not clear from the context, we denote by $\nu_f$ respectively $\nu_s$ the outer unit normal with respect to $Z_f $ respectively $Z_s$.
We   assume that $S^{\pm} \cap \partial Z_f = \emptyset$, \ie the fluid does not touch the upper and lower boundary of the thin layer. Furthermore, we request that $Z_f$ and $Z_s$ are open, connected   with Lipschitz-boundary, and the lateral boundary is $Y$-periodic which means that for $i=1,\ldots,n-1$ and $\ast \in \{s,f\}$
\begin{align*}
\left(\partial Z_{\ast} \cap \{y_i = 0\}\right) + e_i = \partial Z_{\ast} \cap \{y_i=1\}.
\end{align*}
We introduce the set $K_{\vareps}:= \{k \in \Z^{n-1} \times \{0\} \, : \, \vareps(Z + k) \subset \oeps\}$ and we have $\oeps = \mathrm{int}\left(\bigcup_{k\in K_{\vareps}} \vareps(\overline{Z} + k)\right)$.
We define the fluid and solid part of the layer by
\begin{align*}
\oef := \mathrm{int} \left( \bigcup_{k \in K_{\vareps}} \vareps \left(\overline{Z_f} + k \right) \right), \qquad
\oes := \mathrm{int} \left( \bigcup_{k \in K_{\vareps}} \vareps \left(\overline{Z_s} + k \right) \right).
\end{align*}
The fluid-structure interface between the solid and the fluid part is denoted by 
\begin{align*}
\geps := \mathrm{int}\left( \overline{\oes} \cap \overline{\oef}\right).
\end{align*}
The upper and lower boundary of the thin layer is denoted by
\begin{align*}
\sepm := \Sigma \times \{\pm \vareps\}, \qquad S_{\vareps}:= S_{\vareps}^+ \cup S_{\vareps}^-.
\end{align*}
Finally we define a Dirichlet- and Neumann boundary part of the lateral boundary for the fluid. For $i=1,\ldots,n-1$ we define
\begin{align*}
\partial_{a_i} \Sigma:= \{\x \in \partial \Sigma\, : \, \x_i = a_i\}, \quad \partial_{b_i} \Sigma:= \{\x \in \partial \Sigma\, : \, \x_i = b_i\}.
\end{align*}
Let $I_a,I_b\subset \{1,\ldots,n-1\}$ (possibly empty) and define 
\begin{align*}
\partial_D \Sigma := \mathrm{int}\left( \bigcup_{i \in I_a} \partial_{a_i} \Sigma \cup \bigcup_{i\in I_b} \partial_{b_i}\Sigma\right), \qquad \partial_N \Sigma := \mathrm{int} \left(\partial \Sigma \setminus \partial_D \Sigma\right).
\end{align*}
We assume $\partial_N \Sigma \neq \emptyset$ and $\partial_D \Sigma\neq \emptyset$. In particular, we have $|\partial_D \Sigma|>0$ and $|\partial_N \Sigma|>0$. Now, we define the associated microscopic lateral boundaries for the fluid by
\begin{align*}
\partial_D \oef := \mathrm{int}\left( \partial \oef \cap \left( \partial_D \Sigma \times (-\vareps , \vareps) \right) \right),\qquad \partial_N \oef := \mathrm{int}\left( \partial \oef \cap \left( \partial_N \Sigma \times (-\vareps , \vareps) \right) \right).
\end{align*}
For the displacement in the solid we only consider Dirichlet-boundary conditions, and therefore define the whole lateral boundary as Dirichlet-boundary:
\begin{align*}
\partial_D \oes := \mathrm{int}\left( \partial \oes \cap \left(\partial \Sigma \times (-\vareps,\vareps)\right)\right).
\end{align*}

\subsection{Microscopic equations and weak formulation}
We consider the following quasi-static model:  We are looking for a fluid velocity $\veps: (0,T)\times \oef \rightarrow \R^n$, a fluid pressure $\peps: (0,T)\times \oef \rightarrow \R$, and a displacement $\ueps: (0,T)\times \oes \rightarrow \R^n$, such that
\begin{subequations}\label{MicroModel}
\begin{align}
\label{MicroModel_PDE_fluid}
- \vareps \nabla \cdot (D(\veps)) +  \nabla \peps &= f_{\vareps} &\mbox{ in }& (0,T)\times \oef,
\\
\nabla \cdot \veps &= 0 &\mbox{ in }& (0,T)\times \oef,
\\
\label{MicroModel_PDE_solid}
- \frac{1}{\vareps}\nabla \cdot (AD(\ueps)) &= g_{\vareps} &\mbox{ in }& (0,T)\times \oes.
\end{align}
On the interface $\geps$ between the fluid and the solid we assume continuity of the velocities and stresses:
\begin{align}
\partial_t \ueps &= \veps &\mbox{ on }& (0,T)\times \geps,
\\
(-\vareps D(\veps) + \peps I)\nu &= -\frac{1}{\vareps} A D(\ueps)\nu &\mbox{ on }& (0,T)\times \geps,
\end{align}
On the lateral boundary of $\oeps$ we have
\begin{align}
\ueps &= 0 &\mbox{ on }& (0,T)\times \partial_D \oes,
\\
-\frac{1}{\vareps} A D(\ueps) \nu &=0  &\mbox{ on }& (0,T)\times \sepm,
\\
\veps &= 0 &\mbox{ on }& (0,T)\times \partial_D \oef,
\\
\label{MicroModel_PressureBC}
(-\vareps D(\veps) + \peps I)\nu &= 0 &\mbox{ on }& (0,T)\times \partial_N \oef.
\end{align}
Further, an initial condition for the displacement $\ueps$ is needed, where we assume
\begin{align}
\ueps(0)  = 0 \quad \mbox{ in } \oes.
\end{align}
\end{subequations}
We emphasize that the only non-trivial data are the external forces $f_{\vareps}$ and $g_{\vareps}$ in $\eqref{MicroModel_PDE_fluid}$ resp. $\eqref{MicroModel_PDE_solid}$. However, the methods developed in this paper allow to treat more general cases. In particular, we  mention inhomogeneous pressure boundary conditions on the lateral boundary $\partial_N \oef$ in $\eqref{MicroModel_PressureBC}$. Such boundary conditions were treated in \cite{fabricius2023homogenization} for the pure Stokes problem with a slip boundary condition on the perforations.

\begin{definition}[Weak solution]
We call the triple $(\veps,\peps,\ueps)$ a weak solution of the microscopic problem $\eqref{MicroModel}$, if 
\begin{align*}
\veps &\in L^2((0,T),H^1(\oef,\partial_D \oef))^n, \quad \peps \in L^2((0,T)\times \oef),
\\
\ueps &\in H^1((0,T),H^1(\oes,\partial_D \oes))^n,
\end{align*}
with $\nabla \cdot \veps = 0$ and $\partial_t \ueps = \veps$ on $\geps$ and $\ueps(0) = 0$. Further, for all $\phieps \in H^1(\oeps,\partial_D \oeps)^n$ it holds almost everywhere in $(0,T)$
\begin{align}
\begin{aligned}\label{eq:micro_model}
\vareps \int_{\oef} D(\veps): D(\phieps) dx - \int_{\oef} & \peps \nabla \cdot \phieps dx + \frac{1}{\vareps} \int_{\oes} AD(\ueps):D(\phieps) dx 
\\
&= \int_{\oef} f_{\vareps} \cdot \phieps dx + \int_{\oes}g_{\vareps} \cdot \phieps dx .
\end{aligned}
\end{align}
\end{definition}
In the formulation of a weak solution we request $H^1$-regularity with respect to time to guarantee that the initial condition of $\ueps$ and the trace of $\partial_t \ueps$ is well-defined.
%
%
For the homogenization, more precisely to obtain good $\vareps$-uniform a priori estimates, we will need even more regularity with respect to time.
\\

The aim of this paper is the derivation of a macroscopic model including effective coefficients for $\vareps \to 0$, when the thin layer reduces to the lower dimensional submanifold $\Sigma$. The principal idea is to assume that the microscopic solution fulfills a two-scale ansatz. We illustrate this ansatz for the fluid velocity in the layer:
\begin{align}\label{TwoScaleAnsatz}
\veps(t,x) = v_0\tbxfxe + \vareps v_1 \tbxfxe +  \ldots,
\end{align}
with functions $v_j$ which are $Y$-periodic with respect to the variable $y=\frac{x}{\vareps}$ (of course, not periodic in the last component $y_n = \frac{x_n}{\vareps}$). The two-scale convergence gives a rigorous justification of the expansion  in $\eqref{TwoScaleAnsatz}$. We will identify the expansion for the displacement up to order 2, whereas for the fluid velocity and pressure we get the terms up to order 1 (two-scale convergence for the pressure only for order zero).

\subsection{Assumptions on the data}

For the definition of the two-scale convergence we refer to Appendix \ref{sec:two_scale_convergence}.
\begin{enumerate}
[label = (A\arabic*)]
\item The elasticity tensor $A\in \R^{n\times n \times n \times n}$  is symmetric and positive:
\begin{align*}
A_{ijkl} = A_{jilk} = A_{ljik}\quad \mbox{ for all } i,j,k,l=1,\ldots,n,
\\
A B : B \geq c_0 |B|^2 \quad \mbox{ with } c_0>0, \, \forall B \in \R^{n\times n} \mbox{ symmetric}.
\end{align*}

\item\label{ass:f_eps} $f_{\vareps} \in H^4((0,T),L^2(\oef))^n\cap W^{3,\infty}((0,T),L^2(\oef))^n$ with $f_{\vareps}(0) = 0$ and 
\begin{align*}
\|f_{\vareps}\|_{H^3((0,T),L^2(\oef))} + \|f_{\vareps}\|_{W^{2,\infty}((0,T),L^2(\oef))} &\le C \sqrt{\vareps},
\\
\|f_{\vareps}^n\|_{H^3((0,T),L^2(\oef))} + \|f_{\vareps}^n\|_{W^{2,\infty}((0,T),L^2(\oef))} &\le C \vareps^{\frac32}.
\end{align*}
Further, there exists $f_0 \in L^2((0,T)\times \Sigma)^n$ with $f_0^n = 0$ and $f_1^n \in L^2((0,T)\times \Sigma \times Z_f)$, such that $\chi_{\oef}f_{\vareps} \rats \chi_{Z_f}f_0$ and $\chi_{\oef} \vareps^{-1} f_{\vareps}^n \rats \chi_{Z_f} f_1^n$.
\item\label{ass:g_eps}  $g_{\vareps} \in H^4((0,T),L^2(\oes))^n\cap W^{3,\infty}((0,T),L^2(\oes))^n$ with $g_{\vareps}(0) = 0$ and 
\begin{align*}
\|g_{\vareps}\|_{H^3((0,T),L^2(\oes))} + \|g_{\vareps}\|_{W^{2,\infty}((0,T),L^2(\oes))} &\le C \sqrt{\vareps},
\\
\|g_{\vareps}^n\|_{H^3((0,T),L^2(\oes))} + \|g_{\vareps}^n\|_{W^{2,\infty}((0,T),L^2(\oes))} &\le C \vareps^{\frac32}.
\end{align*}
Further, there exists $g_0 \in L^2((0,T)\times \Sigma)^n$ with $g_0^n = 0$ and $g_1^n \in L^2((0,T)\times \Sigma \times Z_s)$, such that $\chi_{\oes}g_{\vareps} \rats \chi_{Z_s}g_0$ and $\chi_{\oes}  \vareps^{-1}g_{\vareps}^n \rats \chi_{Z_s} g_1^n$.
\end{enumerate}
 We emphasize that the $n$-th component of $f_{\vareps}$ and $g_{\vareps}$ behaves different as the first $(n-1)$ components and is of one $\vareps$-order smaller.  Also we mention that the $H^4\cap W^{3,\infty}$-regularity with respect to time for $f_{\vareps}$ and $g_{\vareps}$ is  necessary to establish $H^3$-regularity with respect to time of $\ueps$. A uniform bound with respect to $\vareps$ is only necessary for lower order derivatives.
%

\section{Main results and the Biot-law}
\label{sec:main_results}

Let us formulate the main results of the paper. First of all, we have the following existence, uniqueness, and regularity result:

\begin{theorem}\label{thm:Existence}
There exists a unique weak solution $(\veps,\peps,\ueps)$ of the microscopic problem $\eqref{MicroModel}$. This solution
 fulfills
\begin{align*}
 \veps &\in H^3((0,T),H^1(\oef))^n, \quad 
 \peps \in H^3((0,T),L^2(\oef)),
 \\
 \ueps &\in W^{3,\infty}((0,T),H^1(\oes))^n.
\end{align*}
Additionally, it holds the initial condition $\veps(0) =0$ and $\peps(0) = 0$.
\end{theorem}

\begin{remark}
To obtain existence of a weak solution it would be enough to assume $H^2$ and $W^{1,\infty}$ for $g_{\vareps}$ and $H^1$ for $f_{\vareps}$ with respect to time. This guarantees $\peps \in H^1((0,T),L^2(\oef))$, which is the necessary regularity for the limit pressure $p_0$. However, to obtain suitable $\vareps$-uniform a priori estimates (in particular for $\veps$ and $\partial_t \ueps$), higher regularity for the micro-solution is necessary, see Section \ref{sec:Apriori_Estimates} for details.
\end{remark}

The sequence of microscopic solutions $(\veps,\peps,\ueps)$ converges in a suitable sense, see Section \ref{sec:compactness_results} for the precise results, to limit functions $(v_0,v_1,p_0,\widetilde{u}_1,u_0^n)$. The zeroth and first order fluid velocity correctores $v_0$ and $v_1$ can be expressed in terms of $p_0$, $\widetilde{u}_1$, and $u_0^n$. The latter quantities are the unique weak solutions of a Biot-plate-equation. This macroscopic model will be formulated in the following Theorem (we refer to Section \ref{sec:derivation_macro_model} for the definition of the effective coefficients $B^1,B^2,\alpha^h,K,d_n^f,d_n^s,a^{\ast},b^{\ast},c^{\ast}$ and the averaged quantities $\overline{f_1^n}$ and $\overline{g_1^n}$): 

\begin{theorem}\label{thm:macro_model}
The limit functions $(p_0,\widetilde{u}_1,u_0^n)$ from Proposition \ref{prop:conv_displacement} and \ref{prop:conv_pressure} are the unique weak solutions of the following Biot-plate-model:
$p_0$ solves the generalized Darcy-equation
\begin{subequations}\label{Macro_Model}
\begin{align}
\label{Macro_Model:Darcy_law}
\alpha^h \partial_t p_0 + \nabla_{\x} \cdot \left(K(f_0 - \nabla_{\x} p_0)\right) \\
=\partial_t \big[ \big(B^1 - |Z_f|I\big):D_{\x}(\widetilde{u}_1)& + \big( B^2 + d_n^f I\big) : \nabla_{\x}^2 u_0^n\big]  &\mbox{ in }& (0,T)\times \Sigma,
\notag
\\
K(f_0 - \nabla_{\x} p_0) \cdot \nu &= 0 &\mbox{ on }& (0,T)\times \partial_D \Sigma,
\\
p_0 &= 0 &\mbox{ on }& (0,T)\times \partial_N \Sigma.
\end{align}
The limit functions $u_0^n$ and $\widetilde{u}_1$  solve the  plate equation
\begin{align}
\label{eq:Macro_Model_elastic}
-\nabla_{\x} \cdot \left(a^{\ast} D_{\x}(\widetilde{u}_1) + b^{\ast} \nabla_{\x}^2 u_0^n  + p_0 \left[ B^1 - |Z_f|I\right] \right) &=  |Z_f| f_0  +|Z_s| g_0  &\mbox{ in }& (0,T)\times \Sigma,
\\
\label{Macro_Model:plate_4th_order}
\nabla_{\x}^2 : \left(b^{\ast}D_{\x}(\widetilde{u}_1) +   c^{\ast} \nabla_{\x}^2 u_0^n + p_0 \left[B^2 + d_n^f I\right] \right) 
\\
= \overline{f_1^n} + d_n^f \nabla_{\x} \cdot f_0 &+\overline{g_1^n} + d_n^s \nabla_{\x} \cdot g_0 &\mbox{ in }& (0,T)\times \Sigma,
\notag
\\
\widetilde{u}_1 = u_0^n = \nabla_{\x} u_0^n &= 0 &\mbox{ on }& (0,T)\times \Sigma.
\end{align}
\end{subequations}
Additionally, we have the initial condition $(p_0,\widetilde{u}_1,u_0^n)(0) = 0$ in $\Sigma$.
\\

The limit function $v_1$ from Remark \ref{rem:v1_w1} fulfills 
\begin{align*}
v_1 =\sum_{i=1}^n \left(f_0 - \nabla_{\x}p_0\right)_i q_i + \partial_t \widetilde{u}_1 - y_n \partial_t \nabla_{\x} u_0^n,
\end{align*}
with $q_i$ from the cell problem $\eqref{CellProblem_Stokes}$ and the Darcy-velocity $\v_1$ is given by
\begin{align}\label{def:Darcy_velocity}
\v_1 : = \int_{Z_f} v_1 dy = K(f_0 - \nabla_{\x} p_0) +|Z_f| \partial_t \widetilde{u}_1 - d_n^f \partial_t \nabla_{\x} u_0^n
\end{align}
with $d_n^f:= \int_{Z_f} y_n dy$ and the permeability tensor $K$ defined in $\eqref{def:permeability_tensor}$.
\end{theorem}
 
The Darcy-velocity $\v_1$ is zero in the $n$-th component. Further, also for the function $\widetilde{u}_1$  we have $\widetilde{u}_1^n = 0$ and therefore we can consider $\widetilde{u}_1$ as a function in $L^2(\Sigma,H_0^1(\Sigma))^{n-1}$. In particular, using $f_0^n = g_0^n=0$ equation $\eqref{eq:Macro_Model_elastic}$ is a system of $n-1$ equations.

Let us formulate the weak equation for the macroscopic model $\eqref{Macro_Model}$: The triple $(p_0,\widetilde{u}_1,u_0^n)$ is a weak solution of the problem $\eqref{Macro_Model}$ if 
\begin{align*}
p_0 &\in L^2((0,T),H^1(\Sigma,\partial_N \Sigma))\cap H^1((0,T),L^2(\Sigma)),
\\
\widetilde{u}_1 &\in H^1(0,T),H^1_0(\Sigma))^{n-1}, \quad 
u_0^n \in H^1((0,T),H_0^2(\Sigma))
\end{align*}
and for all $\phi \in H^1(\Sigma,\partial_N \Sigma)$ it holds almost everywhere in $(0,T)$ that
\begin{align}
\begin{aligned}\label{eq:weak_form_Darcy}
\int_{\Sigma} \alpha^h \partial_t p_0 \phi d\x + \int_{\Sigma} &K(\nabla_{\x} p_0 - f_0) \cdot \nabla_{\x} \phi d\x
\\
=& \int_{\Sigma} \partial_t \left[ \left(B^1 - |Z_f|I\right):D_{\x}(\widetilde{u}_1) + \left( B^2 + d_n^f I\right) : \nabla_{\x}^2 u_0^n\right] \phi  d\x.
\end{aligned}
\end{align}
Further, for all  $U\in H^1_0(\Sigma)^{n-1} $ and $V \in H_0^2(\Sigma)$ it holds almost everywhere in $(0,T)$
\begin{align}
\begin{aligned}\label{eq:weak_form_plate}
\int_{\Sigma}& a^{\ast} D_{\x}(\widetilde{u}_1) : D_{\x} (U) + b^{\ast} \nabla_{\x}^2 u_0^n : D_{\x} (U) + b^{\ast} D_{\x}(\widetilde{u}_1) : \nabla_{\x}^2 V + c^{\ast} \nabla_{\x}^2 u_0^n : \nabla_{\x}^2 V d\x
\\
&+ \int_{\Sigma} p_0 [B^1 - |Z_f|I]:D_{\x}(U) + p_0[B^2 + d_n^f I]: \nabla_{\x}^2 V d\x
\\
=&\int_{\Sigma} \overline{f_1^n} V + f_0 \cdot \left[|Z_f| U - d_n^f \nabla_{\x} V\right] d\x +\int_{\Sigma} \overline{g_1^n} V + g_0 \cdot \left[|Z_s| U - d_n^s \nabla_{\x} V\right] d\x
\end{aligned}
\end{align}
and it holds the initial condition $(p_0,\widetilde{u}_1,u_0^n)(0) = 0$ in $\Sigma$.
\\

In summary, we obtain that in the topology of the two-scale convergence (except the $n$-th component in the first order corrector of $v_{\vareps,\app}$, see also Remark \ref{rem:v1_w1}), the microscopic solution $(\veps,\peps,\ueps)$ can be approximated by 
\begin{align}
\begin{aligned}\label{def:approximation_micro_solution}
v_{\varepsilon}^{\app}(t,x) &= \partial_t u_0^n(t,\bar{x}) e_n 
\\
&\hspace{2em}+ \vareps \bigg[\sum_{i=1}^n \left(f_0(\x) - \nabla_{\x}p_0(\x)\right)_i q_i\left(\fxe\right) + \partial_t \tilde{u}_1(t,\bar{x}) -\dfrac{x_n}{\vareps} \nabla_{\x} \partial_t u_0^n(t,\bar{x}) \bigg],
\\
p_{\varepsilon}^{\app}(t,x) &= p_0(t,\x),
\\
u_{\varepsilon}^{\app}(t,x) &= u_0^n(t,\bar{x}) e_n + \vareps \left[ \tilde{u}_1(t,\bar{x}) - \dfrac{x_n}{\vareps} \nabla_{\x} u_0^n(t,\bar{x}) \right] + \vareps^2 u_2\left(t,\x,\fxe\right) 
\end{aligned}
\end{align}
with $u_2$ defined in Proposition \ref{prop:corrector_u2}.

\section{Existence and uniqueness of weak solutions}
\label{sec:Existence}

We show existence and uniqueness of a weak solution for the microscopic problem. For this we can fix the scaling parameter $\vareps$ and therefore we drop in this section the index $\vareps$ or put it equal to $1$ in the microscopic equation. More precisely, we use the notation $\Omega:= \oeps$, $\oef:= \Omega_f$, and $\Omega_s:= \oes$. The fluid-solid interface is denoted by $\Gamma:= \geps$, and the lateral boundaries by $\partial_D \Omega_f$, $\partial_N \Omega_f$, and $\partial_D \Omega_S$, respectively. 
Hence, we have to show the existence of a unique triple $(v,p,u)$ such that
\begin{align*}
v\in L^2((0,T),H^1(\Omega_f,\partial_D \Omega_f))^n, \, p\in L^2((0,T)\times \Omega_f),\, u\in H^1((0,T),H^1(\Omega_s,\partial_D \Omega_s))^n,
\end{align*}
with $\nabla\cdot v = 0$ and $v= \partial_t u$ on $\Gamma$, and for all $\phi \in H^1(\Omega,\partial_D \Omega)^n$ it holds almost everywhere in $(0,T)$
\begin{align}
\begin{aligned}\label{eq:var_eq_existence}
\int_{\Omega_f} D(v) : D(\phi) dx - \int_{\Omega_f}& p\nabla\cdot \phi dx + \int_{\Omega_s} AD(u) : D(\phi)dx 
\\
&= \int_{\Omega_f} f \cdot \phi dx + \int_{\Omega_s} g\cdot \phi dx.
\end{aligned}
\end{align}
Further, it holds the initial condition $u(0) = 0$.
Additionally, we show the regularity results with respect to time formulated in Theorem \ref{thm:Existence}. 
\\

Uniqueness follows directly by choosing $f = 0$ and $g=0$ and testing the weak equation with $\partial_tu$ and $v$. To establish the existence of a solution $(v,p,u)$ we use the Galerkin-method. The only critical point (which seems to be non-standard compared to parabolic problems) is that the associated system of ordinary differential equations for the coefficients in the Galerkin-approximation has no full rank. We define
\begin{align*}
H:= \left\{ \phi \in H^1(\Omega, \partial_D \Omega_f \cup \partial_D \Omega_s)^n\, : \, \nabla \cdot \phi = 0 \, \mbox{ in } \Omega_f\right\},
\end{align*}
together with the inner product
\begin{align*}
(\phi,\psi)_H:= \int_{\Omega_f} D(\phi):D(\psi) dx + \int_{\Omega_s} A D(\phi): D(\psi) dx.
\end{align*}
The associated norm is denoted by $\|\cdot\|_H$. We emphasize that $(\cdot , \cdot)_H$ is a norm due to the ellipticity of $A$ and the Korn inequality. It is standard to show, see for example \cite[Chapter 6.5]{EvansPartialDifferentialEquations}, that there exists an orthonormal basis $\{\phi_i\}_{i\in \N}$  of $H$ with respect to the inner product $(\cdot,\cdot)_H$. Further, we define $H_m:= \mathrm{span}\{\phi_1,\ldots,\phi_m\}$. Our aim is to find functions $\alpha_{i,m}:(0,T)\rightarrow \R$ for $m\in \N$ fixed and $i=1,\ldots,m$, such that 
\begin{align*}
v_m (t,x) &= \sum_{i=1}^m \alpha'_{i,m}(t) \phi_i(x) &\mbox{ for a.e.}& (t,x)\in (0,T)\times \Omega_f, \\
u_m(t,x) &= \sum_{i=1}^m \alpha_{i,m}(t) \phi_i(x) &\mbox{ for a.e.}& (t,x)\in (0,T)\times \Omega_s,
\end{align*}
are approximations of the functions $v$ and $u$. For $\alpha_{i,m}$ sufficiently smooth, we immediately obtain $\partial_t u_m = v_m$ on $\Gamma$. As initial conditions for $\alpha_{i,m}$ we need
\begin{align*}
\alpha_{i,m}(0) = 0 \quad \mbox{ for } i=1,\ldots,m.
\end{align*}
Further, $v_m$ and $u_m$ have to solve the following problem: For all $\phi_i$ with $i=1,\ldots,m$ it holds that
\begin{align}\label{eq:vm_um_Galerkin}
\int_{\Omega_f} D(v_m):D(\phi_i) dx + \int_{\Omega_s} AD(u_m):D(\phi_i) dx = \int_{\Omega_f} f\cdot \phi_i dx + \int_{\Omega_s} g\cdot \phi_i dx.
\end{align}
 Obviously, $\eqref{eq:vm_um_Galerkin}$ is also valid for all test-functions in $H_m$.  We define
\begin{align*}
F&:(0,T)\rightarrow \R^m, &F(t) := \left(\int_{\Omega_f} f \cdot \phi_1 dx,\ldots, \int_{\Omega_f} f \cdot \phi_m dx \right),
\end{align*}
and in a similar way $G:(0,T)\rightarrow \R^m$.
By the regularity assumptions of $f$ and $g$ we obtain  $F,G \in H^4(0,T)^m$. Further, we define $B, C \in \R^{m\times m}$ and $\alpha: (0,T)\rightarrow \R^m$ by
\begin{align*}
B_{ij} := \int_{\Omega_f} D(\phi_i) : D(\phi_j) dx, \,\,\, C_{ij}:= \int_{\Omega_s} AD(\phi_i) : D(\phi_j) dx, \quad \alpha:= (\alpha_{1,m} , \ldots,\alpha_{m,m}).
\end{align*}
Here we suppress the index $m$. With an elemental calculation we get from $\eqref{eq:vm_um_Galerkin}$
\begin{align}\label{eq:ODE_Galerkin}
B \alpha' + C\alpha = F + G.
\end{align}
Here we would like to have the invertibility of $B$. However, this can not be guaranteed by the definition. In fact $\{\phi_i\}_{i=1}^m$ is a basis for $H_m$, hence in particular linearly independent. But it seems that this is not enough to guarantee that the $\phi_i$ are linearly independent on $\Omega_f$. In other words, we can have $\phi:= \sum_{i=1}^m \beta_i \phi_i = 0$ on $\Omega_f$ for $\beta_i$ not all equal to zero (of course in this case $\phi \neq 0$ in $\Omega_s$). Hence, we have to argue in different way to obtain the existence of equation $\eqref{eq:ODE_Galerkin}$.
By definition we have 
\begin{align*}
B_{ij} + C_{ij} = (\phi_i , \phi_j)_H = \delta_{ij},
\end{align*}
hence, we have $B + C = E_m$ with the unit matrix $E_m$ in $\R^{m\times n}$. Since $B$ is symmetric, we obtain the existence of an orthogonal matrix $Q\in \mathcal{O}(m)$ and a diagonal matrix $D \in \R^{m\times m}$, such that $B = Q^t D Q$. Especially, we get
\begin{align*}
C = E_m - B = E_m - Q^t D Q.
\end{align*}
With $\alpha^{\ast}:= Q \alpha$ and $h^*:= Q(F+G)$ we obtain from $\eqref{eq:ODE_Galerkin}$ 
\begin{align*}
D(\alpha^{\ast})^{\prime} = (E_m - D) \alpha^{\ast} + h^{\ast}.
\end{align*}
Without loss of generality we assume that for $l\in \{1,\ldots,m\}$ we have $D_{ii} \neq 0$ for $1\le i \le l$ and $D_{ii} = 0$ for $i>l$ (for $m$ large enough $B$ and therefore $D$ is not zero). Hence, for $i=1,\ldots,l$ the functions $\alpha^{\ast}_i $ are given by the ODE
\begin{align*}
D_{ii} (\alpha^{\ast}_i)^{\prime} = (1 - D_{ii}) \alpha^{\ast}_i + h^{\ast}_i, \quad \alpha^{\ast}_i (0) = 0.
\end{align*}
This problem admits a unique solution $\alpha^{\ast}_i \in H^5(0,T)$. Now, for $D_{ii} = 0$ (i.e., $i>l$) we get the algebraic equation
\begin{align*}
\alpha^{\ast}_i = -h_i^{\ast}.
\end{align*}
We emphasize that in this case, we only have the regularity $\alpha^{\ast}_i \in H^4(0,T)$, so we loose one order of differentiability (compared to $i\le l$). Here, it is crucial to have the compatibility condition $h^{\ast}(0) = 0$. Finally, we obtain $\alpha$ by the transformation $\alpha = Q^t \alpha^{\ast} \in H^4(0,T)^m$.

Altogether, we obtain $v_m \in H^3((0,T),H^1(\Omega_f,\partial_D \Omega_f))^n$ with $\nabla \cdot v_m = 0$ and $u_m \in H^4((0,T),H^1(\Omega_s,\partial_D \Omega_s))^n$. Further, $(v_m,u_m)$ solves the problem $\eqref{eq:vm_um_Galerkin}$ with the initial condition $u_m(0) = 0$. Next, we need a priori estimates with respect to $m$. Since similar arguments are needed in more detail later for the $\vareps$-uniform a priori estimates we only sketch the main ideas. As a test-function in $\eqref{eq:vm_um_Galerkin}$ we choose $v_m$ in $\Omega_f$ and $\partial_t u_m $ in $\Omega_s$, to obtain
\begin{align*}
\|D(v_m)\|^2_{L^2(\Omega_f)} + \frac12 \frac{d}{dt} \|\sqrt{A} D(u_m)\|^2_{L^2(\Omega_s)} = \int_{\Omega_f} f\cdot v_m dx + \int_{\Omega_s} g\cdot \partial_t u_m dx.
\end{align*}
Integration with respect to time and using $D(u_m)(0) = 0$, we obtain for all $t\in (0,T)$
\begin{align}
\begin{aligned}\label{eq:apriori_Galerkin}
\|D(v_m)\|^2_{L^2((0,t)\times \Omega_f)} +\frac12 \|&\sqrt{A} D(u_m)(t)\|^2_{L^2(\Omega_s)} 
\\
&= \int_0^t \int_{\Omega_f} f\cdot v_m dx dt + \int_0^t \int_{\Omega_s} g \cdot \partial_t u_m dx dt.
\end{aligned}
\end{align}
For the first term on the right-hand side we obtain for $\theta>0$ arbitrary  and a constant $C_{\theta}>0$ independent of $m$
\begin{align*}
\left| \int_0^t \int_{\Omega_f} f \cdot v_m dx \right| \le C_{\theta} \|f\|^2_{L^2((0,T)\times \Omega_f)} + \theta \|D(v_m)\|^2_{L^2((0,t)\times \Omega_f)}.
\end{align*}
For the second term on the right-hand side in $\eqref{eq:apriori_Galerkin}$ we get with integration by parts with respect to time (and the assumption $g(0) = 0$ in \ref{ass:g_eps}) and the Korn inequality:
\begin{align*}
\int_0^t \int_{\Omega_s} g \cdot \partial_t u_m dx dt &= - \int_0^T \int_{\Omega_s} \partial_t g u_m dx dt + \int_{\Omega_s} g(t,x) u_m(t,x) dx dt
\\
&\le C_{\theta} \left(\|\partial_t g\|^2_{L^2((0,T)\times \Omega_s)} + \|g\|^2_{L^{\infty}((0,T),L^2(\Omega_s))} \right)  
\\
 &\hspace{2em}+ \|D(u_m)  \|_{L^2((0,t)\times \Omega_s)}^2 + \theta \|D(u_m)(t)\|_{L^2(\Omega_s)}^2.
\end{align*}
Choosing $\theta$ small enough  and using the Gronwall inequality, we obtain
\begin{align*}
\|D(v_m)\|_{L^2((0,T)\times \Omega_f)}& + \|D(u_m)\|_{L^{\infty}((0,T),L^2(\Omega_s))} 
\\
&\hspace{-4em}\le C \left( \|\partial_t g\|_{L^2((0,T)\times \Omega_s)} + \|g\|_{L^{\infty}((0,T),L^2(\Omega_s))} + \|f\|_{L^2((0,T)\times \Omega_f)} \right).
\end{align*}
Using again the Korn inequality, we obtain the a priori estimate
\begin{align*}
\|v_m\|_{L^2((0,T),H^1(\Omega))}& + \|u_m\|_{L^{\infty}((0,T),H^1(\Omega))}
\\
&\le  C \left( \|\partial_t g\|_{L^2((0,T)\times \Omega_s)} + \|g\|_{L^{\infty}((0,T),L^2(\Omega_s))} + \|f\|_{L^2((0,T)\times \Omega_f)} \right)
\end{align*}
for a constant $C>0$ independent of $m$. Differentiating equation $\eqref{eq:vm_um_Galerkin}$ with respect to time (possible because of the time-regularity of $v_m$ and $u_m$, as well as $f$ and $g$), and using the same arguments as above, we obtain the estimate
\begin{align*}
\|v_m&\|_{H^3((0,T),H^1(\Omega_f))} + \|u_m\|_{W^{3,\infty}((0,T),H^1(\Omega_s))} 
\\
&\le  C \left( \|g\|_{H^4((0,T),L^2( \Omega_s))} + \|g\|_{W^{3,\infty}((0,T),L^2(\Omega_s))} + \|f\|_{H^3(((0,T),L^2( \Omega_f))} \right).
\end{align*}
Hence, there exists $v \in H^3((0,T),H^1(\Omega_f))^n$ and $u \in W^{3,\infty}((0,T),H^1(\Omega_s))^n$, such that up to a subsequence  for every $q \in [1,\infty)$.
\begin{align*}
v_m &\rightharpoonup v &\mbox{ weakly in }& H^3((0,T),H^1(\Omega_f))^n,
\\
u_m &\rightharpoonup u &\mbox{ weakly in }& W^{3,q}((0,T),H^1(\Omega_s))^n.
\end{align*}
The last convergence is also valid in the weak$^{\ast}$ topology for $q = \infty$. It is easy to check that $v$ and $u$ solve the problem
\begin{align}\label{eq:Galerkin_final}
\int_{\Omega_f} D(v) : D(\phi_i) dx + \int_{\Omega_s} AD(u):D(\phi_i) dx = \int_{\Omega_f} f \cdot \phi_i dx + \int_{\Omega_s} g \cdot \phi_i dx
\end{align}
for all $i \in \N$ and almost everywhere in $(0,T)$. Since $\phi_i$ is a basis of $H$, we obtain by density that this equation is also valid for arbitrary $\phi \in H$ instead of $\phi_i$. This proves the existence of a weak solution. 


It remains to establish the existence of the pressure $p$ and the initial condition.  We define the mapping (we neglect the dependence on the time variable which only acts as an additional parameter)
\begin{align*}
T: H^1(\Omega, \partial_D \Omega_f \cup \partial_D \Omega_s)^n \rightarrow L^2(\Omega_f), \qquad T(\phi) = \nabla \cdot \phi |_{\Omega_f}.
\end{align*}
It is easy to check that this operator is surjective and therefore closed. Further, the kernel of $T$ is given by $\mathcal{N}(T) = H$. We define $F: H^1(\Omega, \partial_D \Omega_f \cup \partial_D \Omega s)^n \rightarrow \R$ by
\begin{align*}
F(\phi) := \int_{\Omega_f} D(v) : D(\phi) dx + \int_{\Omega_s} A D(u):D(\phi)dx - \int_{\Omega_f} f\cdot \phi dx - \int_{\Omega_s} g\cdot \phi dx.
\end{align*}
Hence, $F=0$ on $H$ and the existence of a unique element $p \in L^2((0,T)\times \Omega_f)$ with $F(\phi) = (p,\nabla \cdot \phi)_{L^2(\Omega_f)}$ follows from the closed range theorem. To show that $p \in H^3((0,T),L^2(\Omega_f))$ we differentiate $\eqref{eq:Galerkin_final}$ with respect to $t$ and proceed in the same way as above.

Finally, for the initial condition $p(0) = 0$ and $v(0) = 0$ we notice that $v\in C^0([0,T],H^1(\Omega_f))^n$, $p \in C^0([0,T],L^2(\Omega_f))$, and $u\in C^0([0,T],H^1(\Omega_s))^n$. Hence, the variational equation $\eqref{eq:var_eq_existence}$ is also valid for $t=0$. This means with $u(0) = 0$, $f(0) = 0$, and $g(0)=0$ that for all $\phi \in H^1(\Omega,\partial_D \Omega_f \cup \partial_D \Omega_s)^n$
\begin{align*}
\int_{\Omega_f} D(v)(0) : D(\phi) dx - \int_{\Omega_f} p(0) \nabla \cdot \phi dx =0.
\end{align*}
Choosing $\phi = v(0)$ and then $\phi$ such that $\nabla\cdot \phi = p(0)$, we get $v(0) = 0$ and $p(0) =0$.

\section{A priori estimates for the microscopic solutions}
\label{sec:Apriori_Estimates}

In this section we derive a priori estimates with respect to $\vareps$. For the homogenization we need suitable estimates for $(\veps,\peps,\ueps)$. Here, a critical point is the force term for the fluid velocity and the scaling in the fluid equation. To overcome this problem and to obtain better estimates with respect to $\vareps$ it will be necessary to control the second time derivatives of $\ueps$. For this, we differentiate the microscopic problem $\eqref{MicroModel}$ with respect to time. This leads (due to our initial condition for $\ueps$, $f_{\vareps}$, and $g_{\vareps}$) to the same variational equation for the time derivatives (of course, forces including also the time derivatives). Hence, in the following lemma we introduce an abstract variational problem for which we derive $\vareps$-uniform a priori estimates.

\begin{lemma}\label{lem:abstract_apriori_lem1}
Let $\Veps \in L^2((0,T),H^1(\oef))^n$  and $\Ueps \in H^1((0,T),H^1(\oes))^n$ with   $\nabla \cdot \Veps = 0$ and  $\partial_t \Ueps = \Veps$ on $\geps$, and for all $\phieps \in H^1(\oeps,\partial_D \oeps)^n$ it holds that
\begin{align}\label{eq:abstract_Micro_Model}
\vareps\int_{\oef}D(\Veps): D(\phieps) dx + \frac{1}{\vareps} \int_{\oes} A D(\Ueps) : D(\phieps) dx = \int_{\oef} F_{\vareps} \cdot \phieps dx  + \int_{\oes} G_{\vareps} \cdot \phieps dx
\end{align}
with given $F_{\vareps} \in L^2((0,T)\times \oes)^n$ and $G_{\vareps} \in H^1((0,T),L^2(\oes))^n\cap L^{\infty}((0,T),L^2(\oes))^n$ with $G_{\vareps}(0) = 0$. Further, we assume the initial condition $\Ueps(0)=0$.
Then we have the following a priori estimate with a constant $C>0$ independent of $\vareps$:
\begin{align*}
\sqrt{\vareps}& \|D(\Veps)\|_{L^2((0,T)\times \oef)} + \frac{1}{\sqrt{\vareps}}\|D(\Ueps)\|_{L^{\infty}((0,T),L^2(\oes))}
\\
\le & C\bigg(\sqrt{\vareps}\sum_{i=1}^{n-1} \left\{  \|\partial_t G_{\vareps}^i\|_{L^2((0,T)\times \oes)} +  \|G_{\vareps}^i\|_{L^{\infty}((0,T),L^2(\oes))} \right\} 
+ \frac{1}{\sqrt{\vareps}} \|\partial_t G_{\vareps}^n\|_{L^2((0,T)\times \oes)} 
\\
&+ \frac{1}{\sqrt{\vareps}} \|G_{\vareps}^n\|_{L^{\infty}((0,T),L^2(\oes))} 
+ \frac{1}{\sqrt{\vareps}}\sum_{i=1}^{n-1} \|F_{\vareps}^i \|_{L^2((0,T)\times \oef)} + \frac{1}{\vareps^{\frac32}} \|F_{\vareps}^n\|_{L^2((0,T)\times \oef)}\bigg)
\end{align*}
\end{lemma}
\begin{proof}
Choosing as a test-function in $\eqref{eq:abstract_Micro_Model}$ the function $(\Veps,\partial_t \Ueps)$, we obtain
\begin{align}\label{eq:aux_apriori_abstract_basic}
\vareps \|D(\Veps)\|^2_{L^2(\oef)} + \frac{1}{2\vareps}  \frac{d}{dt} \|\sqrt{A}D(\Ueps)\|^2_{L^2(\oes)} = \int_{\oef} F_{\vareps} \cdot \Veps dx + \int_{\oes} G_{\vareps} \cdot \partial_t \Ueps dx.
\end{align}
We estimate the right-hand side. For the first term we obtain with the Korn inequality from Lemma \ref{lem:app_Korn_thin} in the appendix almost everywhere in $(0,T)$
\begin{align*}
\int_{\oef} F_{\vareps} \cdot \Veps dx &\le \sum_{i=1}^{n-1} \|F_{\vareps}^i \|_{L^2(\oef)} \|\Veps\|_{L^2(\oef)} + \|F_{\vareps}^n \|_{L^2(\oef)} \|\Veps^n \|_{L^2(\oef)}
\\
&\le C \left(\sum_{i=1}^{n-1} \|F_{\vareps}^i \|_{L^2(\oef)} + \frac{1}{\vareps} \|F_{\vareps}^n \|_{L^2(\oef)} \right) \|D(\Veps)\|_{L^2(\oef)}
\\
&\le \frac{C}{\vareps} \left( \sum_{i=1}^{n-1} \|F_{\vareps}^i \|_{L^2(\oef)}^2 + \frac{1}{\vareps^2} \|F_{\vareps}^n\|^2_{L^2(\oef)}\right) + \frac{\vareps}{2}\|D(\Veps)\|^2_{L^2(\oef)}.
\end{align*}
For the term including $G_{\vareps}$ we  integrate with respect to time, use integration by parts and again the Korn inequality to get for almost every $t\in (0,T)$ and arbitrary $\theta>0$
\begin{align}
\begin{aligned}\label{eq:aux_apriori_abstract_Geps}
\int_0^t &\int_{\oes} G_{\vareps} \cdot \partial_t \Ueps dx dt = - \int_0^t \int_{\oes} \partial_t G_{\vareps} \cdot \Ueps dx dt + \int_{\oes} G_{\vareps}(t,x) \cdot \Ueps(t,x) dx
\\
\le& C\left(\sum_{i=1}^{n-1} \|\partial_t G_{\vareps}^i\|_{L^2((0,T)\times \oes)} + \frac{1}{\vareps} \|\partial_t G_{\vareps}^n\|_{L^2((0,T)\times \oes)} \right) \|D(\Ueps)\|_{L^2((0,t)\times \oes)}
\\
&+ C \left(\sum_{i=1}^{n-1} \| G_{\vareps}^i\|_{L^{\infty}((0,T),L^2( \oes))} + \frac{1}{\vareps} \| G_{\vareps}^n\|_{L^\infty((0,T),L^2( \oes))} \right) \|D(\Ueps)\|_{L^2( \oes)}
\\
\le & C\bigg(\sum_{i=1}^{n-1} \left\{ \vareps \|\partial_t G_{\vareps}^i\|^2_{L^2((0,T)\times \oes)} + \vareps \|G_{\vareps}^i\|_{L^{\infty}((0,T),L^2(\oes))}^2 \right\} 
+ \frac{1}{\vareps}\|\partial_t G_{\vareps}^n\|^2_{L^2((0,T)\times \oes)} 
\\
&+ \frac{1}{\vareps} \|G_{\vareps}^n\|_{L^{\infty}((0,T),L^2(\oes))}^2 \bigg)
+ \frac{1}{\vareps} \left(\|D(\Ueps)\|_{L^2((0,t)\times \oes)}^2 +  \theta \|D(\ueps)(t)\|_{L^2(\oes)}^2\right)
\end{aligned}
\end{align}
With the Gronwall inequality, the ellipticity of $A$, and $D(\Ueps)(0) = 0$ we obtain the desired result.
\end{proof}
 Later we will choose $\Veps = \partial^{\alpha}_t \veps$, $\Ueps = \partial_t^{\alpha} \ueps$, $F_{\vareps} = \partial_t^{\alpha} f_{\vareps}$, and $G_{\vareps} = \partial_t^{\alpha} g_{\vareps}$ with $\alpha = 0,1,2$. 
Next, we derive an estimate for $\Veps$ and $\Ueps$ under the assumption that we can control $\partial_t \Ueps$. For this we would like to consider the difference $\veps - \partial_t \ueps$, which vanishes on the interface $\geps$. Since $\partial_t \ueps$ is not defined in the fluid domain $\oef$ we have to extend it in a suitable way. We use the extension from Lemma \ref{lem:Extension_operator} and denote by $\tUeps$ the extension  to the whole layer $\oeps$. More precisely, we define $\tUeps(t,\cdot):= E_{\vareps}\Ueps(t,\cdot)$.

\begin{lemma}\label{lem:abstract_apriori_lem2}
Let the assumptions of Lemma \ref{lem:abstract_apriori_lem1} be valid. Additionally, we assume $F_{\vareps} \in H^1((0,T),L^2(\oef))^n\cap L^{\infty}((0,T),L^2(\oef))^n$ with $F_{\vareps}(0) = 0$. Then, we have the following estimate with a constant $C>0$ independent of $\vareps$:
\begin{align*}
&\sqrt{\vareps} \|D(\Veps)\|_{L^2((0,T)\times \oef)} + \frac{1}{\sqrt{\vareps}} \|D(\Ueps)\|_{L^{\infty}((0,T),L^2(\oes))}
\\
&\le C\bigg(\sum_{i=1}^{n-1} \left\{ \sqrt{\vareps} \|\partial_t G_{\vareps}^i\|_{L^2((0,T)\times \oes)} + \sqrt{\vareps} \|G_{\vareps}^i\|_{L^{\infty}((0,T),L^2(\oes))} \right\} 
\\
&+ \frac{1}{\sqrt{\vareps}}\|\partial_t G_{\vareps}^n\|_{L^2((0,T)\times \oes)} + \frac{1}{\sqrt{\vareps}} \|G_{\vareps}^n\|_{L^{\infty}((0,T),L^2(\oes))} \bigg)
\\
&+ \bigg(\sum_{i=1}^{n-1} \left\{ \sqrt{\vareps} \|\partial_t F_{\vareps}^i\|_{L^2((0,T)\times \oef)} + \sqrt{\vareps} \|F_{\vareps}^i\|_{L^{\infty}((0,T),L^2(\oef))} \right\} 
\\
&+ \frac{1}{\sqrt{\vareps}}\|\partial_t F_{\vareps}^n\|_{L^2((0,T)\times \oef)} + \frac{1}{\sqrt{\vareps}} \|F_{\vareps}^n\|_{L^{\infty}((0,T),L^2(\oef))} \bigg)
+ C \sqrt{\vareps} \|D(\partial_t \Ueps)\|_{L^2((0,T)\times \oes)}.
\end{align*}

\end{lemma}
\begin{proof}
We estimate the force term in the fluid domain in the right-hand side of $\eqref{eq:aux_apriori_abstract_basic}$ (again, in $\eqref{eq:abstract_Micro_Model}$ we choose the test-function $(\Veps,\partial_t \Ueps)$) in the following way (almost everywhere in $(0,T)$): With $\Weps:= \Veps - \partial_t \tUeps$ we have
\begin{align}
\begin{aligned}\label{eq:aux_apriori_FepsDtupes}
\int_{\oef} F_{\vareps} \cdot \Veps dx &= \int_{\oef} F_{\vareps} \cdot \Weps dx + \int_{\oef} F_{\vareps} \cdot \partial_t \tUeps dx.
\end{aligned}
\end{align}
For the first term we notice $\Weps = 0$ on $\geps$, and therefore we can use the Poincar\'e and the Korn inequality from Lemma \ref{lem:ineq_Poincare_Korn} to obtain
\begin{align}
\begin{aligned}\label{eq:aux_Weps_Korn_Poincare}
\|\Weps\|_{L^2(\oef)}&\le C\vareps \|\nabla \Weps\|_{L^2(\oef)} \le C\vareps \|D(\Weps)\|_{L^2(\oef)}
\\
 &\le C\vareps \left( \|D(\Veps)\|_{L^2(\oef)} + \|D(\partial_t \Ueps)\|_{L^2(\oes)}\right),
\end{aligned}
\end{align}
and therefore for all $t\in (0,T)$
\begin{align*}
\int_0^t &\int_{\oef} F_{\vareps} \cdot \Weps dx dt 
\\
&\le C \vareps \|F_{\vareps}\|_{L^2((0,T)\times \oef)}^2 + \frac{\vareps}{2} \|D(\Veps)\|^2_{L^2((0,t)\times \oef)} + C \vareps \|D(\partial_t \Ueps)\|_{L^2((0,t)\times \oes)}^2.
\end{align*}
For the second term on the right-hand side in $\eqref{eq:aux_apriori_FepsDtupes}$ we use the same arguments as in $\eqref{eq:aux_apriori_abstract_Geps}$ together with the estimates for the extension operator from Lemma \ref{lem:Extension_operator} to obtain for $t\in (0,T)$ and arbitrary $\theta>0$
\begin{align*}
\int_0^t&\int_{\oef} F_{\vareps}\cdot \partial_t \tUeps dx  dt = -\int_0^t \int_{\oef} \partial_t F_{\vareps} \cdot \tUeps dx dt + \int_{\oef} F_{\vareps}(t,x) \cdot \tUeps (t,x) dx 
\\
\le & C\left(\sum_{i=1}^{n-1} \|\partial_t F_{\vareps}^i\|_{L^2((0,T)\times \oef)} + \frac{1}{\vareps} \|\partial_t F_{\vareps}^n\|_{L^2((0,T)\times \oef)} \right) \|D(\tUeps)\|_{L^2((0,t)\times \oef)}
\\
&+ C \left(\sum_{i=1}^{n-1} \| F_{\vareps}^i\|_{L^{\infty}((0,T),L^2( \oef))} + \frac{1}{\vareps} \| F_{\vareps}^n\|_{L^\infty((0,T),L^2( \oef))} \right) \|D(\tUeps)\|_{L^2( \oef)}
\\
\le & C\bigg(\sum_{i=1}^{n-1} \left\{ \vareps \|\partial_t F_{\vareps}^i\|^2_{L^2((0,T)\times \oef)} + \vareps \|F_{\vareps}^i\|_{L^{\infty}((0,T),L^2(\oef))}^2 \right\} 
+ \frac{1}{\vareps}\|\partial_t F_{\vareps}^n\|^2_{L^2((0,T)\times \oef)} 
\\
&+\frac{1}{\vareps} \|F_{\vareps}^n\|_{L^{\infty}((0,T),L^2(\oef))}^2 \bigg)
+ \frac{1}{\vareps} \left(\|D(\Ueps)\|_{L^2((0,t)\times \oes)}^2 +  \theta \|D(\Ueps)(t)\|_{L^2(\oes)}^2\right).
\end{align*}
Altogether, we obtain with $\eqref{eq:aux_apriori_abstract_basic}$ after integration with respect to time from $0$ to $t \in (0,T)$ and $\eqref{eq:aux_apriori_abstract_Geps}$,  choosing $\theta$ small enough,  the positivity of $A$, and the Gronwall inequality
\begin{align*}
\vareps \|&D(\Veps)\|^2_{L^2((0,T)\times \oef)} + \frac{1}{\vareps} \|D(\Ueps)\|^2_{L^{\infty}((0,T),L^2(\oes))}
\\
\le& C\bigg(\sum_{i=1}^{n-1} \left\{ \vareps \|\partial_t G_{\vareps}^i\|^2_{L^2((0,T)\times \oes)} + \vareps \|G_{\vareps}^i\|_{L^{\infty}((0,T),L^2(\oes))}^2 \right\} 
\\
&+ \frac{1}{\vareps}\|\partial_t G_{\vareps}^n\|^2_{L^2((0,T)\times \oes)} + \frac{1}{\vareps} \|G_{\vareps}^n\|_{L^{\infty}((0,T),L^2(\oes))}^2 \bigg)
\\
&+ \left(\sum_{i=1}^{n-1} \bigg\{ \vareps \|\partial_t F_{\vareps}^i\|^2_{L^2((0,T)\times \oef)} + \vareps \|F_{\vareps}^i\|_{L^{\infty}((0,T),L^2(\oef))}^2 \right\} 
\\
&+ \frac{1}{\vareps} \|\partial_t F_{\vareps}^n\|^2_{L^2((0,T)\times \oef)} + \frac{1}{\vareps} \|F_{\vareps}^n\|_{L^{\infty}((0,T),L^2(\oef))}^2 \bigg)
+ C \vareps \|D(\partial_t \Ueps)\|_{L^2((0,t)\times \oes)}^2.
\end{align*}
\end{proof}
In contrast to Lemma \ref{lem:abstract_apriori_lem1} we obtain that the terms including $F_{\vareps}$ are of one order better. The price to pay is that the time derivative $\partial_t \Ueps$ occurs. However, we will see that for $\partial_t^{\alpha} \veps$ and $\partial_t^{\alpha} \ueps$ with $\alpha = 0,1$, we can improve in this way the $\vareps$-uniform bound compared to the second time derivatives $\partial_{tt} \veps$ and $\partial_{tt} \ueps$. 

In the following lemma we give the a priori estimates for $\veps$ and $\ueps$. For the displacement we consider the extension to the whole layer given by Lemma \ref{lem:Extension_operator} from the appendix. We define for almost every $t \in (0,T)$
\begin{align}\label{def:tildeUeps}
\tueps(t,\cdot_x):= E_{\vareps}\ueps(t,\cdot_x).
\end{align}
\begin{lemma}\label{lem:apriori_velo_displac}
The functions $\veps$ and $\weps$ in the weak solution $(\veps,\peps,\ueps)$ of the microscopic problem \ref{MicroModel} fulfill
\begin{align*}
\frac{1}{\sqrt{\vareps}}\|D( \veps)\|_{H^1((0,T),L^2( \oef))} + \frac{1}{\vareps^{\frac32}} \|D( \ueps)\|_{W^{1,\infty}((0,T),L^2(\oes))} \le C.
\end{align*}
Further, for $\weps:= \veps - \partial_t \tueps \in H^1(\oef,\geps)^n$ it holds that
\begin{align*}
\frac{1}{\vareps^{\frac32}} \|\weps \|_{L^2((0,T)\times \oef)} + \frac{1}{\sqrt{\vareps}} \|\nabla \weps\|_{L^2((0,T)\times \oef)} + \frac{1}{\vareps^{\frac32}} \|\nabla \cdot \weps \|_{L^2((0,T)\times \oef)} \le C.
\end{align*}
In particular, it holds for $i=1,\ldots,n-1$ that
\begin{align*}
\|\veps^i\|_{H^1((0,T),L^2(\oef))} \le C\vareps^{\frac32}.
\end{align*}
\end{lemma}
\begin{proof}
We first notice that the abstract variational equation $\eqref{eq:abstract_Micro_Model}$ is exactly the equation for $\veps$ and $\ueps$ from the definition of the weak microscopic solution with $F_{\vareps} = f_{\vareps}$ and $G_{\vareps} = g_{\vareps}$. Differentiating the microscopic problem $\eqref{MicroModel}$ with respect to time, taking into account the initial condition $\ueps(0)=0$ and the compatibility condition $f_{\vareps} (0) = 0$ and $g_{\vareps}(0) = 0$ (see assumptions \ref{ass:f_eps} and \ref{ass:g_eps}), the associated weak variational equation for $\partial_t^{\alpha} \veps $ and $\partial_t^{\alpha} \ueps$ with $\alpha =1,2$ are $\eqref{eq:abstract_Micro_Model}$ with $F_{\vareps} = \partial_t^{\alpha} f_{\vareps} $ and $G_{\vareps} = \partial_t^{\alpha} g_{\vareps}$ (and $\Veps = \partial_t^{\alpha} \veps$ and $\Ueps = \partial_t^{\alpha} \ueps$). 
Hence, from Lemma \ref{lem:abstract_apriori_lem1} (here the $H^3$-regularity with respect to time for $\ueps$ is necessary) and the assumptions on $f_{\vareps} $ and $g_{\vareps}$, we  get for $\alpha = 0,1,2$:
\begin{align}\label{eq:aux_apriori_final}
\sqrt{\vareps}\|D(\partial_t^{\alpha} \veps)\|_{L^2((0,T)\times \oef)} + \frac{1}{\sqrt{\vareps}} \|D(\partial_t^{\alpha} \ueps)\|_{L^{\infty}((0,T),L^2(\oes))} \le C.
\end{align}
We emphasize that the right-hand side is of order $1$ and not $\vareps$, due to norms of $f_{\vareps}$. Now, we use Lemma \ref{lem:abstract_apriori_lem2} to obtain for $\beta = 0,1$
\begin{align*}
\sqrt{\vareps}\|D(\partial_t^{\beta} \veps)\|_{L^2((0,T)\times \oef)} + \frac{1}{\sqrt{\vareps}} \|&D(\partial_t^{\beta} \ueps)\|_{L^{\infty}((0,T),L^2(\oes))} 
\\
&\le C\vareps + C\sqrt{\vareps}\|D(\partial_t^{\beta + 1} \ueps)\|_{L^2((0,T)\times \oes)}
\le C\vareps,
\end{align*}
where in the second inequality we used $\eqref{eq:aux_apriori_final}$. 

Next, we estimate the norms of $\weps$ and $\veps^i$. We use the same arguments as in $\eqref{eq:aux_Weps_Korn_Poincare}$ in the proof of Lemma \ref{lem:abstract_apriori_lem2}, in particular it holds that $\weps = 0$ on $\geps$. From the Poincar\'e and Korn inequality from Lemma \ref{lem:ineq_Poincare_Korn} in the appendix we get
\begin{align*}
\|\weps\|_{L^2((0,T)\times \oef)} \le C \vareps \|\nabla \weps \|_{L^2((0,T)\times \oef)} \le  C \vareps \|D(\weps)\|_{L^2((0,T)\times \oef)} \le C \vareps^{\frac{3}{2}},
\end{align*}
where at the end we used the a priori estimates for $\veps$ and $\partial_t \ueps$ obtained above together with the properties of the extension operator from Lemma \ref{lem:Extension_operator}. To estimate $\veps^i$ with $i=1,\ldots,n-1$ we use the Korn inequality in Lemma \ref{lem:app_Korn_thin}, the extension operator from Lemma \ref{lem:Extension_operator}, and the bound for $D(\partial_t \ueps)$ obtained above to get
\begin{align*}
\|\veps^i\|_{L^2((0,T)\times \oef)} &\le \|\weps^i\|_{L^2((0,T)\times \oef)} + \|\partial_t \tueps^i\|_{L^2((0,T)\times \oef)}
\\
&\le C\vareps^{\frac32} + C \left( \|\partial_t \ueps^i\|_{L^2((0,T)\times \oes)}  + \varepsilon \|\nabla \ueps \|_{L^2((0,T)\times \oes)}\right)
\\
&\le C \varepsilon^{\frac32} + \|D(\ueps)\|_{L^2((0,T)\times \oes)} \le C\vareps^{\frac32}.
\end{align*} 
It remains to estimate the norm of $\nabla \cdot \weps$. Since $\veps $ is divergence free, we have
\begin{align*}
\nabla \cdot \weps = -\nabla \cdot \partial_t \tueps = - \mathrm{tr}(D(\partial_t \tueps)).
\end{align*}
From the estimate of $D(\partial_t \ueps)$ and again Lemma \ref{lem:Extension_operator} we obtain the desired result.
\end{proof}
It remains to estimate the fluid pressure $\peps$, where we use the Bogovskii-operator from Corollary \ref{cor:Bogovskii} in the appendix.
\begin{lemma}\label{lem:apriori_pressure}
The pressure $\peps$ in the weak microscopic problem $\eqref{MicroModel}$ fulfills
\begin{align*}
\|\peps\|_{H^1((0,T),L^2(\oef))} \le C\sqrt{\vareps},
\end{align*}
for a constant $C>0$ independent of $\vareps$.
\end{lemma}
\begin{proof}
From Corollary \ref{cor:Bogovskii} we obtain the existence of $\phieps \in H^1(\oef, \geps \cup \partial_D \oef)^n$ with $\nabla \cdot \phieps = -\peps$ (almost everywhere in $(0,T)$)  and
\begin{align*}
\vareps \|\nabla \phieps\|_{L^2(\oef)} \le C \|\peps\|_{L^2(\oef)}.
\end{align*}
We extend $\phieps$ by zero to the whole layer $\oeps$ and plug in this function as a test-function in $\eqref{eq:micro_model}$ to obtain (with the Poincar\'e inequality form Lemma \ref{lem:ineq_Poincare_Korn} in the appendix)
\begin{align*}
\|\peps\|_{L^2(\oef)}^2  &= -\vareps \int_{\oef} D(\veps) : D(\phieps)dx + \int_{\oef}f_{\vareps} \cdot \phieps dx
\\
&\le \vareps \|D(\veps)\|_{L^2(\oef)} \|D(\phieps)\|_{L^2(\oef)} + \|f_{\vareps}\|_{L^2(\oef)} \|\phieps\|_{L^2(\oeps)}
\\
&\le C \left(\|D(\veps)\|_{L^2(\oef)} + \sqrt{\vareps}\right)\|\peps\|_{L^2(\oef)}.
\end{align*}
Integration with respect to time and using the a priori estimate for $D(\veps)$ from Lemma \ref{lem:apriori_velo_displac} gives the estimate for $\peps$ in $L^2((0,T)\times \oef)$. To obtain the estimate for $\partial_t \peps$ we formally differentiate $\eqref{eq:micro_model}$ with respect to time (this can be made rigorous in the Galerkin-method in Section \ref{sec:Existence}) and argue in the same way as above.
\end{proof}

\section{Compactness microscopic solutions}
\label{sec:compactness_results}

The aim of this section is the derivation of compactness results for the microscopic solution $(\veps,\peps,\ueps)$. As the underlying topology we use the method of two-scale convergence, see Appendix \ref{sec:two_scale_convergence} for a brief introduction and summary of the main two-scale compactness results related to our setting. We emphasize that the convergence results in this section are purely based on the a priori estimates in the previous Section \ref{sec:Apriori_Estimates} and basic properties (boundary conditions and $\nabla\cdot \veps = 0$). We do not use explicitly that $(\veps,\peps,\ueps)$ is in fact a weak microscopic solution of $\eqref{MicroModel}$.

In the following result we use the extension $\tueps$ defined in $\eqref{def:tildeUeps}$. We emphasize that
 Lemma \ref{lem:Extension_operator} immediately implies that $\tueps$ fulfills the same a priori estimates as $\ueps$ in Lemma \ref{lem:apriori_velo_displac}. 


\begin{proposition}\label{prop:conv_displacement}
Let $\ueps$ be the displacement from the micro-solution $(\veps,\peps,\ueps)$ of  the microscopic model $\eqref{MicroModel}$ with extension $\tueps$.
Then there exist limit functions $u_0^n \in H^1((0,T),H_0^2(\Sigma))\cap H^1((0,T),L^2(\Sigma))$ and $\tilde{u}_1 \in H^1((0,T),H^1_0(\Sigma))^n$ with $\tilde{u}_1^n = 0$, and $u_2 \in H^1((0,T), L^2( \Sigma, H_{\#}^1(Z)/\R))^n$, such that up to a subsequence (for $\alpha = 1,\ldots,n-1$ and $\beta = 0,1$)
\begin{align*}
\frac{\partial_t^{\beta}\tueps^{\alpha}}{\vareps} &\rats \partial_t^{\beta}\tilde{u}_1^{\alpha} - y_n \partial_t^{\beta}\partial_{\alpha} u_0^n,
\\
 \partial_t^{\beta}\tueps^n &\rats  \partial_t^{\beta}u_0^n,
\\
\foe  D(\partial_t^{\beta}\tueps) &\rats  D_{\x}(\partial_t^{\beta}\tilde{u}_1) - y_n \partial_t^{\beta}\nabla_{\x}^2 u_0^n + \partial_t^{\beta}D_y(u_2).
\end{align*}
The limit functions $\widetilde{u}_1$ and $u_0^n$ fulfill the initial conditions $\widetilde{u}_1(0) = 0$ and $u_0^n(0) = 0$.
Further, it holds up to a subsequence that
\begin{align*}
\partial_t^{\beta}\ueps\vert_{\geps} \rats \partial_t^{\beta}u_0^n e_n \qquad\mbox{ on } \geps.
\end{align*}
\end{proposition}
\begin{proof}
This result is an immediate consequence of Lemma \ref{lem:app_TS_Plate} from the appendix. We also refer to \cite[Proposition 5.2]{gahn2022derivation} for more details, in particular for the two-scale convergence on the surface.
\end{proof}


From the a priori estimates in Section \ref{sec:Apriori_Estimates} we see that the bounds for the fluid velocity $\veps$ differ from the bounds of $\ueps$ resp. $\partial_t \ueps$, leading to a different behavior in the limit. We also have seen in Lemma \ref{lem:apriori_velo_displac} that we get better estimates with respect to $\varepsilon$ for the function $\weps= \veps - \partial_t \tueps$, so it is also suitable to work with this function.
\begin{proposition}\label{prop:conv_fluid}
Let $\veps$ and $\peps$ be the fluid velocity and the fluid pressure from the microscopic solution $(\veps,\peps,\ueps)$ of $\eqref{MicroModel}$, and $\weps:= \veps - \partial_t \tueps$. With the limit functions obtained in Proposition \ref{prop:conv_displacement}, we obtain the following result. There exist $w_1 \in L^2((0,T)\times \Sigma , H_{\#}^1(Z_f,\Gamma))^n$ with $\nabla_y \cdot w_1 = 0$ and $p_0 \in L^2((0,T)\times \Sigma \times Z_f)$, such that up to a subsequence for $i=1,\ldots,n-1$
\begin{align*}
\chi_{\oef} \frac{\veps^i}{\vareps} &\rats \chi_{Z_f} \left[w_1^i + \partial_t \widetilde{u}_1^i - y_n \partial_t \partial_i u_0^n \right],
\\
\chi_{\oef}\veps^n &\rats \chi_{Z_f} \partial_t u_0^n,
\\
\chi_{\oef}D(\veps) &\rats \chi_{Z_f} D_y(w_1),
\end{align*}
\end{proposition}

\begin{proof}
From the a priori estimates in Lemma \ref{lem:apriori_velo_displac} and the two-scale compactness results from Lemma \ref{lem:app_ts_comp_general} in the appendix we obtain the existence of $v_0 \in L^2((0,T)\times \Sigma \times Z_f)^n$ with $v_0^i=0$ for $i=1,\ldots,n-1$, $w_1 \in L^2((0,T)\times \Sigma, H_{\#}^1(Z_f,\Gamma))^n$ (see Lemma \ref{TwoScaleCompactnessPerforated} for the zero boundary value), and $v_1 \in L^2((0,T)\times \Sigma \times Z_f)^{n-1}$ such that up to a subsequence for $i=1,\ldots,n-1$
\begin{align*}
\chi_{\oef} \frac{\weps}{\vareps} &\rats \chi_{Z_f} w_1, \quad 
\chi_{\oef} \vareps \nabla \frac{ \weps}{\vareps} \rats \chi_{Z_f} \nabla_y w_1,
\\
\chi_{\oef} \veps &\rats \chi_{Z_f} v_0, \quad 
\chi_{\oef} \frac{\veps^i}{\vareps} \rats \chi_{Z_f} v_1^i.
\end{align*}
With the convergence results for $\partial_t \tueps$ from Proposition \ref{prop:conv_displacement} we immediately obtain  
\begin{align*}
\chi_{\oef}D(\veps) = \chi_{\oef}\left[\vareps D\left(\frac{\weps}{\vareps}\right) + D(\partial_t \tueps)\right] \rats \chi_{Z_f} D_y(w_1)
\end{align*}
In a similar way we obtain for $i=1,\ldots,n-1$
\begin{align*}
v_1^i = w_1^i + \partial_t \widetilde{u}_1^i - y_n \partial_t \partial_i u_0^n \qquad\mbox{in } (0,T)\times \Sigma \times Z_f.
\end{align*}
Obviously, we have $v_0^i=0 $ for $i=1,\ldots,n-1$. The identity $v_0^n = \partial_t u_0^n$ can be obtained in the same way as in \cite[Proposition 5.3]{gahn2022derivation}. However, here we give a much simpler argument using the function $\weps$. It holds that 
\begin{align*}
\chi_{\oef}\veps^n =  \chi_{\oef}\left(\weps^n + \partial_t \tueps^n \right) \rats \chi_{Z_f} \partial_t u_0^n.
\end{align*}
\end{proof}

\begin{remark}\label{rem:v1_w1}
From the proof above we have for $i=1,\ldots,n-1$ that $\chi_{\oef} \vareps^{-1} \veps^i \rats \chi_{Z_f} v_1^i$. If we define $v_1^n:= w_1^n$ we have 
\begin{align*}
v_1 = w_1 + \partial_t \widetilde{u}_1 - y_n \partial_t \nabla_{\x}u_0^n.
\end{align*}
Since $ \veps \rats \chi_{Z_f} \partial_t u_0^n e_n$ the zeroth order corrector $v_0$ in $\eqref{TwoScaleAnsatz}$ is given by $v_0 = \partial_t u_0^n e_n$. Hence, the two-scale convergence of $\chi_{\oef} \vareps^{-1} \veps^i$ of the first $n-1$ components implies that $v_1^i$ corresponds to the first order corrector $v_1$ in $\eqref{TwoScaleAnsatz}$. However, our convergence results obtained above do not guarantee that for the $n$-th component $w_n$. Nevertheless, this has no influence on the macroscopic model.
\end{remark}
Next, we consider the two-scale limit of $\vareps^{-1} \nabla \cdot \weps$. In particular, we obtain in the limit an equation for the divergence of the mean value $\w_1$ from which we will later obtain the generalized Darcy law.
\begin{proposition}\label{prop:conv_divergence}
Let $\weps = \veps - \partial_t \tueps$ and $(u_0^n,\widetilde{u}_1, u_2)$ the limit functions from Proposition \ref{prop:conv_displacement}. Up to a subsequence it holds that
\begin{align*}
\chi_{\oef} \nabla \cdot \frac{\weps}{\vareps} &\rats \chi_{Z_f} \left[-\nabla_{\x} \cdot \partial_t \widetilde{u}_1 + y_n \Delta_{\x} \partial_t u_0^n -\nabla_y \cdot \partial_t u_2\right].
\end{align*}
In particular we obtain for $\w_1:= \int_{Z_f} w_1 dy$ that  $\nabla_{\x} \cdot \w_1 \in L^2((0,T)\times \Sigma)$ with
\begin{align*}
\nabla_{\x} \cdot \w_1 &= -|Z_f| \nabla_{\x} \cdot \partial_t \widetilde{u}_1 + d_n^f \Delta_{\x} \partial_t u_0^n + \int_{\Gamma} \partial_t u_2 \cdot \nu d\sigma
\\
&= -|Z_f| \nabla_{\x} \cdot \partial_t \widetilde{u}_1 + d_n^f \Delta_{\x} \partial_t u_0^n + \int_{Z_s} \nabla_y \cdot \partial_t u_2 dy -  \int_{S} \partial_t u_2 \cdot \nu d\sigma,
\\
\w_1 \cdot \nu &= 0  \quad\mbox{ on } (0,T)\times \partial_D \Sigma,
\end{align*}
for $d_n^f:= \int_{Z_f} y_n dy$.
See $\eqref{def:normal_trace_zero_part_boundary}$ for the definition of the zero normal trace on a part of the boundary.
\end{proposition}
\begin{proof}
$\nabla_y \cdot w_1 = 0$ was already shown in Proposition \ref{prop:conv_fluid}. Next, we use the convergence results for $D(\partial_t \tueps)$ from Proposition \ref{prop:conv_displacement} and $\nabla \cdot \veps = 0$ to obtain
\begin{align*}
\chi_{\oef}\nabla \cdot \frac{\weps}{\vareps} = - \frac{1}{\vareps}\chi_{\oef} \mathrm{tr} D(\partial_t \tueps) \rats - \chi_{Z_f} \mathrm{tr}\left[ D_{\x} (\partial_t \widetilde{u}_1) - y_n \nabla_{\x}^2 \partial_t u_0^n + D_y(\partial_t u_2)\right]
\end{align*}
which implies the desired convergence result. It remains to show the equation for the divergence of the mean of $w_1$. From the convergence results obtained above we get for all $\phi \in C_0^{\infty}((0,T)\times \overline{\Sigma} \setminus \partial_N \Sigma)$ (using $\weps = -\partial_t \tueps$ on $\partial_D \oef$)
\begin{align*}
\int_0^T \int_{\Sigma} \int_{Z_f} & w_1 \cdot \nabla_{\x} \phi dy d\x dt = \lim_{\vareps \to 0} \foe \int_0^T\int_{\oef} \frac{\weps}{\vareps} \cdot \nabla_{\x} \phi dx dt 
\\
&=- \lim_{\vareps \to 0}\left[ \foe \int_0^T \int_{\oef} \nabla \cdot \frac{\weps}{\vareps} \phi dx dt + \frac{1}{\vareps^2} \int_0^T \int_{\partial_D\oef} \partial_t \tueps \cdot \nu \phi d\sigma_x dt\right].
\end{align*}
We emphasize that on $\partial_N \oef$ the function $\phi$ vanishes and therefore we only have an integral over $\partial_D \oef$ in the integration by parts above. For the boundary term we have with Lemma \ref{lem:app_trace_lateral_BC}, Lemma \ref{lem:Extension_operator}, and the a priori estimates from Lemma \ref{lem:apriori_velo_displac}
\begin{align*}
\left| \frac{1}{\vareps^2} \int_0^T \int_{\partial_D \oef} \partial_t \tueps \cdot \nu \phi d\sigma_x dt \right| &\le \frac{1}{\vareps^2} \|\partial_t \tueps\|_{L^2((0,T)\times \partial_D \oef)}\|\phi\|_{L^2((0,T)\times \partial_D \oef)}
\\
&\le \frac{C}{\vareps} \|D(\partial_t \tueps)\|_{L^2((0,T)\times \oeps)} \le C\sqrt{\vareps}.
\end{align*}
Hence, together with the two-scale convergence for $\nabla \cdot \frac{\weps}{\vareps}$ obtained above we get
\begin{align*}
\int_{\Sigma} \int_{Z_f} w_1 \cdot \nabla_{\x} \phi dy d\x  = \int_{\Sigma} \int_{Z_f} \left[ \nabla_{\x} \cdot \partial_t \widetilde{u}_1  -y_n \Delta_{\x} \partial_t u_0^n + \nabla_y \cdot \partial_t u_2\right] \phi dy d\x 
\end{align*}
almost everywhere in $(0,T)$.
We consider the term including $\partial_t u_2$ (we can neglect the time derivative which commutes with the integral):
\begin{align*}
\int_{Z_f} \nabla_y \cdot u_2 dy &= \int_{\Gamma} u_2 \cdot \nu_f d\sigma = -\int_{\partial Z_s} u_2 \cdot \nu_s d\sigma +  \int_{S} u_2 \cdot \nu_s d\sigma
\\
&= -\int_{Z_s} \nabla_y \cdot u_2 dy + \int_{S} u_2 \cdot \nu_s d\sigma.
\end{align*}
This implies the desired result for the divergence of the mean value of $w_1$. We emphasize that the zero normal trace condition (see $\eqref{def:normal_trace_zero_part_boundary}$ for the definition) follows, since $C_0^{\infty}(\overline{\Sigma}\setminus \partial_N \Sigma)$ is dense in $H^1(\Sigma,\partial_N \Sigma)$, see \cite{Bernard2011}. 
\end{proof}
%

\begin{proposition}\label{prop:conv_pressure}
Let $\peps$ be the fluid pressure from the micro-solution $(\veps,\peps,\ueps)$ of $\eqref{MicroModel}$. Then, there exists $p_0 \in H^1((0,T),L^2( \Sigma \times Z_f))$ such that up to a subsequence
\begin{align*}
\chi_{\oef}\peps \rats  \chi_{Z_f} p_0,\qquad \chi_{\oef} \partial_t \peps \rats \chi_{Z_f} \partial_t p_0.
\end{align*}
Further, $p_0$ fulfills the initial condition $p_0(0) = 0$.
\end{proposition}
\begin{proof}
The existence of $p_0$ and the two-scale convergence of $\peps$ is a direct consequence of  Lemma \ref{lem:apriori_pressure} and \ref{lem:app_ts_comp_general}. The initial condition $p_0(0) = 0$ follows from $\peps(0) = 0$.
\end{proof}

\section{Derivation of the Biot-Plate-Model}
\label{sec:derivation_macro_model}

In this section we derive the macroscopic model $\eqref{Macro_Model}$.
The model consists of two problems: A generalized Darcy-law $\eqref{Macro_Model:Darcy_law}$ and a plate equation $\eqref{eq:Macro_Model_elastic}$ and $\eqref{Macro_Model:plate_4th_order}$ together with suitable boundary and initial conditions. Of course, both problems are strongly coupled. In a first step we give a characterization of the corrector $u_2$ from Proposition \ref{prop:conv_displacement} in terms of $\widetilde{u}_1$ and $u_0^n$ together with several cell solutions. We also formulate here associated cell problems and define homogenized coefficients needed later for the formulation of the plate equation. Based on the representation of $u_2$ and the convergence results in the previous section we are able to derive the Darcy-law. Finally, we show that $(\widetilde{u}_1,u_0^n)$ solves the plate equation.

We start with the definition of suitable cell problems in the solid reference element $Z_s$. For this, we define the symmetric matrices $M_{ij} \in \R^{n\times n}$ for $i,j=1,\ldots,n$ by
\begin{align*}
M_{ij} = \frac{e_i \otimes e_j}{2} + \frac{e_j \otimes e_i}{2}.
\end{align*}
Further, we define  $\chi_{ij} \in H^1_{\#}(Z_s)^n$ as the solutions of the cell problems 
\begin{align}
\begin{aligned}\label{CellProblemStandard}
-\nabla_y \cdot (A (D_y(\chi_{ij}) + M_{ij})) &= 0 &\mbox{ in }& Z_s,
\\
-A(D_y(\chi_{ij}) + M_{ij} ) \nu &= 0 &\mbox{ on }& \Gamma \cup S,
\\
\chi_{ij} \mbox{ is } Y\mbox{-periodic, } & \int_{Z_s} \chi_{ij} dy = 0.
\end{aligned}
\end{align}
Due to the Korn-inequality, this problem has a unique weak solution. We emphasize again that the only rigid-displacements on $Z_f$, which are $Y$-periodic, are constants. Obviously, for $i=1,\ldots,n$ we have $\chi_{in} = \chi_{ni} = y_n e_i - \frac{1}{|Z_s|}\int_{Z_s} y_n e_i dy$.

Additionally, we define $\chi_{ij}^B \in H^1_{\#}(Z_s)^n$ as the solutions of the cell problems
\begin{align}
\begin{aligned}\label{CellProblemHesse}
-\nabla_y \cdot \left( A(D_y(\chi_{ij}^B) - y_n M_{ij})\right) &= 0 &\mbox{ in }& Z_s,
\\
-A(D_y(\chi_{ij}^B) - y_n M_{ij})\nu &= 0 &\mbox{ on }& \Gamma \cup S,
\\
\chi_{ij}^B \mbox{ is } Y\mbox{-periodic, } & \int_{Z_s} \chi_{ij}^B dy = 0.
\end{aligned}
\end{align}
In the same way as above we obtain the existence of a unique weak solution. We also have for $i=1,\ldots,n$ that $\chi_{in}^B = \frac12 y_n^2 e_i - \frac{1}{2|Z_s|}\int_{Z_s} y_n^2 e_i dy$. Further, we denote by $\chi_0 \in H_{\#}^1(Z_s)^n$ the solution of the cell problem
\begin{align}
\begin{aligned}\label{CellProblemPressure}
-\nabla_y \cdot (A D_y(\chi_0)) &= 0 &\mbox{ in }& Z_s,
\\
-AD_y(\chi_0)\nu &=  \nu &\mbox{ on }& \Gamma ,
\\
-AD_y(\chi_0)\nu &=  0 &\mbox{ on }&  S<,
\\
\chi_0 \mbox{ is } Y\mbox{-periodic, } & \int_{Z_s} \chi_0 dy = 0.
\end{aligned}
\end{align}
The weak formulation for this problem reads as follows: $\chi_0 \in H_{\#}^1(Z_s)^n/\R^n$ solves for all $\phi \in H_{\#}^1(Z_s)^n$ (we emphasize that $\int_{\Gamma} \nu d\sigma = 0$) the problem
\begin{align}\label{eq:cell_weak_pressure}
\int_{Z_s} AD_y(\chi_0) : D_y(\phi) dy + \int_{\Gamma} \phi \cdot \nu d\sigma = 0.
\end{align}
After an integration by parts, this can be written as 
\begin{align*}
\int_{Z_s} AD_y(\chi_0) : D_y(\phi) dy + \int_{Z_s} \nabla_y \cdot \phi dy -  \int_S \phi\cdot \nu d\sigma = 0.
\end{align*}
We would like to figure out at this point the difference to problems in the full domain (not in a thin layer), where here we have an additional boundary term. Compare for example the results in \cite{brun2018upscaling}.
Next, we define the effective elasticity coefficients $a^{\ast},\, b^{\ast},\, c^{\ast} \in \R^{(n-1)\times (n-1) \times (n-1) \times (n-1)}$  for  $i,j,k,l = 1,\ldots,n-1$ by
\begin{align}
\begin{aligned}\label{def:homo_elastic_coefficients}
a^{\ast}_{ijkl} &:=  \frac{1}{\vert Z_s \vert} \int_{Z_s} A  \left(D_y(\chi_{i j}) + M_{ij}\right): \left(D_y (\chi_{kl})  + M_{kl} \right)dy,
\\
b^{\ast}_{ijkl} &:=   \frac{1}{\vert Z_s \vert} \int_{Z_s} A \left(D_y(\chi_{i j}^B)  - y_n M_{ij} \right) : \left(D_y (\chi_{kl})  + M_{kl} \right)dy,
\\
c^{\ast}_{ijkl} &:=   \frac{1}{\vert Z_s \vert} \int_{Z_s} A  \left(D_y(\chi_{i j}^B)  - y_n M_{ij} \right): \left(D_y (\chi_{kl}^B)  -y_n M_{kl}\right)dy,
\end{aligned}
\end{align}
with the cell solutions $\chi_{ij}$ and $\chi_{ij}^B$ defined above. The tensors $a^{\ast}$, $b^{\ast}$, and $c^{\ast}$ are positive in the following sense (see \cite[p. 195]{griso2020homogenization}): There exists a constant $c>0$ such that for all $\tau,\sigma \in \R^{(n-1)\times (n-1)} $ symmetric it holds that
\begin{align}\label{ineq:positivity_abc}
a^{\ast} \tau : \tau + b^{\ast}\tau : \sigma + b^{\ast} \sigma :\tau + c^{\ast} \sigma : \sigma \geq c \left[ \|\tau\|^2 + \|\sigma\|^2\right].
\end{align} 
Next, we define  $B^1 ,B^2\in \R^{(n-1)\times (n-1)}$   for $i,j=1,\ldots,n-1$ via
\begin{align}\label{def:homo_pressure_coefficients}
B^1_{ij} := \int_{Z_s} AD_y(\chi_0) :M_{ij} dy ,\qquad B^2_{ij} := -\int_{Z_s} AD_y(\chi_0): M_{ij} y_n dy
\end{align}
with the cell solution $\chi_0$ defined in $\eqref{CellProblemPressure}$ above. Obviously $B^1$ and $B^2$ are both symmetric. Now, we have the following representation of $u_2$.
\begin{proposition}\label{prop:corrector_u2}
The limit function $u_2$ from Proposition \ref{prop:conv_displacement} fulfills for almost every $(t,\x,y) \in (0,T)\times \Sigma \times Z_s$ 
\begin{align*}
u_2(t,\x,y) = \sum_{i,j=1}^{n-1} \left[ D_{\x}(\tilde{u}_1)_{ij} (t,\x) \chi_{ij}(y) +  \partial_{ij}u_0^n(t,\x) \chi_{ij}^B(y) \right] + p_0(t,\x) \chi_0(y),
\end{align*}
where the cell solutions $\chi_{ij}$, $\chi_{ij}^B$, and $\chi_0$ are defined in $\eqref{CellProblemStandard}$, $\eqref{CellProblemHesse}$, and $\eqref{CellProblemPressure}$.
\end{proposition}
\begin{proof}
As a test-function in $\eqref{MicroModel}$ we choose $\phieps(t,x):= \phi\btxfxe $ with $\phi \in C_0^{\infty}((0,T)\times \Sigma, C_{\#}^{\infty}(Z))^n$ to obtain
\begin{align*}
\int_{\oef}&  D(\veps) : \left[ \vareps D_{\x} (\phi) + D_y(\phi)\right]\btxfxe dx - \foe \int_{\oef} \peps \left[ \vareps \nabla_{\x} \cdot  \phi + \nabla_y \cdot \phi\right] \btxfxe dx
\\
+& \foe \int_{\oes} A \foe D(\ueps) :  \left[ \vareps D_{\x} (\phi) + D_y(\phi)\right]\btxfxe dx 
\\
&= \int_{\oef} f_{\vareps} \cdot \phi \btxfxe dx + \int_{\oes} g_{\vareps} \cdot \phi \btxfxe dx.
\end{align*}
After integration with respect to time we pass to the limit $\vareps \to 0$. From the a priori estimates in Lemma \ref{lem:apriori_velo_displac} it follows that the first term on the left-hand side vanishes. Due to our assumptions on $f_{\vareps}$ and $g_{\vareps}$ the same holds for the right-hand side. From the convergence results in Proposition \ref{prop:conv_displacement} and \ref{prop:conv_pressure} we get (the terms including $\nabla_{\x} \cdot \phi$ and $D_{\x}(\phi)$ vanish) almost everywhere in $(0,T)\times \Sigma$
\begin{align*}
0 =  \int_{Z_s} A\left[D_{\x}(\widetilde{u}_1) - y_n \nabla_{\x}^2 u_0^n + D_y(u_2)\right] : D_y(\phi) dy  -  \int_{Z_f} p_0 \nabla_y \cdot \phi dy .
\end{align*}
In other words, $u_2$ is the weak solution of the problem
\begin{align*}
-\nabla_y \cdot \left(A\left[D_{\x}(\widetilde{u}_1) - y_n \nabla_{\x}^2 u_0^n + D_y(u_2)\right] \right) &= 0 &\mbox{ in }& (0,T)\times \Sigma \times Z_s,
\\
-A\left[D_{\x}(\widetilde{u}_1) - y_n \nabla_{\x}^2 u_0^n + D_y(u_2)\right]\nu &= p_0 \nu &\mbox{ on }& (0,T)\times \Sigma \times \Gamma,
\\
-A\left[D_{\x}(\widetilde{u}_1) - y_n \nabla_{\x}^2 u_0^n + D_y(u_2)\right]\nu &= 0 &\mbox{ on }& (0,T)\times \Sigma \times S,
\\
u_2 \mbox{ is } Y\mbox{-periodic}, \, \int_{Z_s} u_2 dy = 0.
\end{align*}
By the Korn inequality this problem has a unique weak solution. By an elemental calculation we obtain the desired result.
\end{proof}

\subsection{The generalized Darcy-law}

We derive the generalized Darcy-law $\eqref{Macro_Model:Darcy_law}$ and characterize the Darcy-velocity $\v_1$ (see $\eqref{def:Darcy_velocity}$). First of all, we derive a two-pressure Stokes model including the limit functions $w_1$ and $p_0$, as well as an additional pressure $p_1$. From this model we obtain a representation of $w_1$ and $\w_1$, and therefore also for $v_1$ and $\v_1$. Then we can easily use the representation of $\nabla_{\x} \cdot \w_1$ obtained in Proposition \ref{prop:conv_divergence} to obtain the generalized Darcy-law.

Before we give the two-pressure Stokes model, we formulate the usual Stokes cell problem on $Z_f$: The tuple $(q_i,\pi_i) \in H^1_{\#}(Z_f,\Gamma)^n\times L^2(Z_f)/\R$ is the unique weak solution of
\begin{align}
\begin{aligned}\label{CellProblem_Stokes}
-\nabla_y \cdot D(q_i) + \nabla_y \pi_i &= e_i &\mbox{ in }& Z_f,
\\
\nabla_y \cdot q_i &= 0 &\mbox{ in }& Z_f,
\\
q_i &= 0 &\mbox{ on }& \Gamma,
\\
q_i,\, \pi_i \mbox{ are } Y\mbox{-periodic}.
\end{aligned}
\end{align}
We emphasize that $(q_n,\pi_n) = (0,y_n)$.
Further, we define the permeability tensor $K \in \R^{n\times n}$ for $i,j=1,\ldots,n$ by
\begin{align}\label{def:permeability_tensor}
K_{ij} := \int_{Z_f} D_y(q_i) : D_y(q_j) dy,
\end{align}
with $q_i$ defined in $\eqref{CellProblem_Stokes}$.
Obviously, we have $K_{ni} = K_{in} = 0$ for $i=1,\ldots,n$, so $K$ can also be regarded as a matrix in $\R^{(n-1)\times (n-1)}$ (we use the same notation). It is well-known that $K$ is symmetric and positive on $\R^{n-1}$.
Next, we observe that the limit pressure $p_0$ is independent of the micro-variable $y$:
\begin{proposition}\label{prop:p_0}
The macroscopic pressure $p_0$ from Proposition \ref{prop:conv_pressure} fulfills for almost all $(t,\x,y) \in (0,T)\times \Sigma \times Z_f$ 
\begin{align*}
p_0(t,\x,y) = p_0(t,\x).
\end{align*}
\end{proposition}
\begin{proof}
As a test-function in $\eqref{MicroModel}$ we choose $\phieps(t,x):= \phi\btxfxe $ with $\phi \in C_0^{\infty}((0,T)\times \Sigma \times  Z_f)^n$. After integration with respect to time we can pass to the limit $\vareps \to 0$ and obtain (we refer to the proof of Proposition \ref{prop:corrector_u2} for more details)
\begin{align*}
- \int_0^T \int_{\Sigma} \int_{Z_f} p_0(t,\x,y) \nabla_y \cdot \phi(t,\x,y) dy d\x dt = 0.
\end{align*}
This implies that $p_0$ is constant with respect to $y$.
\end{proof}
Now, we are able to derive the two-pressure Stokes model:
\begin{proposition}\label{prop:cell_equation_w1_Darcy_vel}
We have $p_0 \in L^2((0,T),H^1(\Sigma,\partial_N \Sigma))$ and there exists $p_1 \in L^2((0,T)\times \Sigma \times Z_f)$, such that for all $\phi \in L^2(\Sigma, H_{\#}^1(\Sigma, \partial_N \Sigma))^n$ and almost everywhere in $(0,T)$ it holds that
\begin{align*}
\int_{\Sigma} \int_{Z_f} D_y(w_1) : D_y(\phi)& dy d\x  + \int_{\Sigma} \int_{Z_f} \nabla_{\x} \tilde{p}_0 \cdot \phi dy d\x 
\\
&-\int_{\Sigma} \int_{Z_f} p_1 \nabla_y \cdot \phi dy d\x = |Z_f| \int_{\Sigma} f_0 \cdot \phi d\x .
\end{align*}
Further, $w_1$ and $p_1$ have the representation
\begin{align*}
w_1 = \sum_{i=1}^n \left(f_0 - \nabla_{\x}p_0\right)_i q_i, \qquad p_1 = \sum_{i=1}^n \left(f_0 -\nabla_{\x}p_0\right)_i \pi_i.
\end{align*}
In particular, the Darcy-velocity $\w_1$ is given by
\begin{align*}
\w_1 = K (f_0 - \nabla_{\x} p_0)
\end{align*}
with the permeability tensor $K$ defined in $\eqref{def:permeability_tensor}$.
\end{proposition}
\begin{proof}
We test the microscopic equation $\eqref{eq:micro_model}$ with $\phieps(t,x):= \vareps^{-1}\phi\btxfxe $ with  $\phi \in C^{\infty}_0 ((0,T)\times \overline{\Sigma}\setminus \partial_D\Sigma , H_{\#}^1(Z_f,\Gamma))^n$ with $\nabla_y \cdot \phi = 0$ (see $\eqref{def:Hper_Gamma}$ for the definition of $H_{\#}^1(Z_f,\Gamma)$).
Here we extend $\phi$ by zero to the whole cell $Z$. We get after an integration with respect to time
\begin{align*}
\frac{1}{\vareps} \int_0^T\int_{\oef}&  D(\veps) : \left[ \vareps D_{\x} (\phi) + D_y(\phi)\right]\btxfxe dx dt
\\
& - \foe \int_0^T\int_{\oef} \peps  \nabla_{\x} \cdot \phi\btxfxe dx dt 
= \foe \int_0^T\int_{\oef} f_{\vareps} \cdot \phi \btxfxe dxdt.
\end{align*}
Using the two-scale compactness results for $D(\veps)$ and $\peps$ from Proposition \ref{prop:conv_fluid} and \ref{prop:conv_pressure}, and the assumption \ref{ass:f_eps} for $f_{\vareps}$, we obtain for $\vareps \to 0$ ($\overline{\phi}$ denotes the mean of $\phi$ with respect to $Z_f$, see also $\eqref{def:mean_Z_f}$)
\begin{align*}
\int_0^T\int_{\Sigma} \int_{Z_f} D_y(w_1) : D_y(\phi) dy d\x dt - \int_0^T \int_{\Sigma} p_0  \nabla_{\x} \cdot \overline{\phi}d\x dt = |Z_f| \int_0^T \int_{\Sigma} f_0 \cdot \phi d\x dt.
\end{align*}
By the density result in Lemma \ref{lem:density_V} this equation is valid for all $\phi \in L^2((0,T), \spaceV)$, see $\eqref{def:spaceV}$ for the definition of $\spaceV$.
From Lemma \ref{lem:two_pressure_decom} we obtain the existence of $\tilde{p}_0\in L^2((0,T),H^1(\Sigma,\partial_N \Sigma))$ and $p_1 \in L^2((0,T)\times \Sigma \times Z_f)$, such that almost everywhere in $(0,T)$ it holds that
\begin{align*}
\int_{\Sigma}\left\{ \int_{Z_f} D_y(w_1) : D_y(\phi) dy   +   \nabla_{\x} \tilde{p}_0 \cdot \overline{\phi}  - \int_{Z_f} p_1 \nabla_y \cdot \phi dy \right\} d\x = |Z_f| \int_{\Sigma} f_0 \cdot \phi d\x 
\end{align*}
for all $\phi \in L^2((0,T)\times \Sigma, H_{\#}^1(Z_f,\Gamma))^n$. By the uniqueness of the decomposition in Lemma \ref{lem:two_pressure_decom} we immediately obtain $p_0 = \tilde{p}_0$.  In other words, the tuple $(w_1,p_1)$ is the unique weak solution of the cell problem (see also Proposition \ref{prop:conv_fluid})
\begin{align*}
-\nabla_y \cdot D_y(w_1)   + \nabla_y p_1 &= f_0 - \nabla_{\x} p_0 &\mbox{ in }& (0,T)\times \Sigma \times Z_f,
\\
\nabla_y \cdot w_1 &= 0 &\mbox{ in }& (0,T)\times \Sigma \times Z_f,
\\
w_1 &= 0 &\mbox{ on }& (0,T)\times \Sigma \times \Gamma,
\\
w_1,\, p_1 \mbox{ are } Y\mbox{-periodic}.
\end{align*}
Using some standard arguments we obtain the desired result.
\end{proof}

\begin{remark}
With (see also Remark \ref{rem:v1_w1})
\begin{align*}
v_1 := w_1 + \partial_t \widetilde{u}_1 - y_n \nabla_{\x} \partial_t u_0^n
\end{align*}
we have that $(v_1,p_1)$ is the unique weak solution of 
\begin{align*}
-\nabla_y \cdot D_y(v_1)   + \nabla_y p_1 &= f_0 - \nabla_{\x} p_0 &\mbox{ in }& (0,T)\times \Sigma \times Z_f,
\\
\nabla_y \cdot v_1 &= 0 &\mbox{ in }& (0,T)\times \Sigma \times Z_f,
\\
v_1 &= \partial_t \widetilde{u}_1 - y_n \nabla_{\x} \partial_t u_0^n &\mbox{ on }& (0,T)\times \Sigma \times \Gamma,
\\
v_1,\, p_1 \mbox{ are } Y\mbox{-periodic}.
\end{align*}
Further, we have with $d_n^f:= \int_{Z_f} y_n dy$
\begin{align*}
\overline{v}_1 = K(f_0 - \nabla_{\x} p_0) +|Z_f| \partial_t \widetilde{u}_1 - d_n^f \partial_t \nabla_{\x} u_0^n.
\end{align*}
\end{remark}
Finally, we want to derive the Darcy-law for $p_0$. For this we use the identity for $\nabla_{\x} \cdot \w_1$ from Proposition \ref{prop:conv_divergence}. Together with the representation of $u_2$ from Proposition \ref{prop:corrector_u2} and the identity for $\w_1$ form Proposition \ref{prop:cell_equation_w1_Darcy_vel} we get
\begin{align*}
\nabla_{\x} \cdot \left(K(f_0 - \nabla_{\x} p_0)\right) =& \nabla_{\x} \cdot \w_1
\\
=&\partial_t \left( -|Z_f| \nabla_{\x} \cdot \widetilde{u}_1 + d_n^f \Delta_{\x} u_0^n \right) + \partial_t p_0 \int_{\Gamma} \chi_0 \cdot \nu d\sigma_y
\\
&+ \sum_{i,j=1}^{n-1} \left\{ \partial_t D_{\x} (\widetilde{u}_1)_{ij} \int_{\Gamma} \chi_{ij} \cdot \nu d\sigma_y + \partial_t \partial_{ij} u_0^n \int_{\Gamma} \chi_{ij}^B \cdot \nu d\sigma_y \right\}.
\end{align*}
We introduce the scalar $\alpha^h \in \R$  
\begin{align*}
\alpha^h :=- \int_{\Gamma} \chi_0 \cdot \nu d\sigma_y.
\end{align*}
By testing the weak equation $\eqref{eq:cell_weak_pressure}$ for the cell problem $\eqref{CellProblemPressure}$ with $\chi_0$  we obtain that $\alpha^h> 0$.
Further, we obtain with $B^1$ from $\eqref{def:homo_pressure_coefficients}$ using the cell problems $\eqref{CellProblemStandard}$ and $\eqref{CellProblemPressure}$ for $\chi_{ij}$ and $\chi_0$ (for $i,j=1,\ldots,n-1$)
\begin{align*}
\int_{\Gamma} \chi_{ij} \cdot \nu d\sigma &=  - \int_{Z_s} AD_y(\chi_0) : D_y(\chi_{ij}) dy = \int_{Z_s} AM_{ij} : D_y(\chi_0)  dy = B_{ij}^1.
\end{align*}
With similar arguments we obtain
\begin{align*}
\int_{\Gamma} \chi_{ij}^B \cdot \nu d\sigma = - \int_{Z_s} A D_y(\chi_{ij}^B)  : D_y(\chi_0) dy
= - \int_{Z_s} AD_y(\chi_0) :M_{ij} y_n dy = B_{ij}^2
\end{align*}
with $B^2 $ from $\eqref{def:homo_pressure_coefficients}$.
Hence, we obtain
\begin{align*}
\nabla_{\x} \cdot \left(K(f_0 - \nabla_{\x} p_0)\right) =&\partial_t \left( -|Z_f| \nabla_{\x} \cdot \widetilde{u}_1 + d_n^f \Delta_{\x} u_0^n \right) - \alpha^h \partial_t p_0 
\\
&+ \partial_t \left( D_{\x}(\widetilde{u}_1) : B^1 + \nabla_{\x}^2 u_0^n : B^2 \right)
\\
=& \partial_t \left[ \left(B^1 - |Z_f|I\right):D_{\x}(\widetilde{u}_1) + \left( B^2 + d_n^f I\right) : \nabla_{\x}^2 u_0^n\right] - \alpha^h \partial_t p_0.
\end{align*}
This is equation  $\eqref{Macro_Model:Darcy_law}$.

\subsection{The plate equation}
To finish the proof of Theorem \ref{thm:macro_model}, we have to show the plate equation $\eqref{eq:Macro_Model_elastic}$ - $\eqref{Macro_Model:plate_4th_order}$. For this, we use a test-function  in $\eqref{eq:micro_model}$ with the same structure as in the asymptotic expansion $\ueps^{\app}$ in $\eqref{def:approximation_micro_solution}$ (up to order $\vareps$). More precisely, we choose 
\begin{align*}
\phieps(t,x):= \frac{1}{\vareps^2} \psi(t) \left[ V(\x) e_n + \vareps\left( U(\x) - \frac{x_n}{\vareps} \nabla_{\x} V(\x)\right)\right]
\end{align*}
with $\psi\in C_0^{\infty}(0,T)$, $U\in H^1_0(\Sigma)^{n-1} \times \{0\}$ (hence $U^n=0$), and $V \in H_0^2(\Sigma)$. We have
\begin{align*}
D(\phieps) = \frac{1}{\vareps} \psi(t) \left[ D_{\x}(U) - \frac{x_n}{\vareps} \nabla_{\x}^2 V\right], \qquad
\nabla \cdot \phieps = \frac{1}{\vareps} \psi(t) \left[ \nabla_{\x} \cdot U - \frac{x_n}{\vareps} \Delta_{\x} V\right].
\end{align*}
Hence, we obtain for $\eqref{eq:micro_model}$  almost everywhere in  $ (0,T)$
\begin{align*}
 \int_{\oef}& D(\veps) : \left[  D_{\x} (U)  - \frac{x_n}{\vareps} \nabla_{\x}^2 V\right]  \psi dx 
- \foe\int_{\oef} \peps \left[ \nabla_{\x} \cdot U - \frac{x_n}{\vareps} \Delta_{\x} V\right] \psi dx 
\\
&+ \foe \int_{\oes} A\frac{D(\ueps)}{\vareps} : \left[ D_{\x} (U) - \frac{x_n}{\vareps} \nabla_{\x}^2 V\right]  \psi dx 
\\
=& \foe  \int_{\oef} f_{\vareps} \cdot \left[\foe V e_n + U - \frac{x_n}{\vareps} \nabla_{\x} V \right] \psi  dx 
+ \foe \int_{\oes} g_{\vareps} \cdot \left[\foe V e_n + U - \frac{x_n}{\vareps} \nabla_{\x} V \right] \psi  dx .
\end{align*}
Integration with respect to time and using the convergence results from Proposition \ref{prop:conv_displacement}, \ref{prop:conv_fluid}, and \ref{prop:conv_pressure}, as well as the assumptions \ref{ass:f_eps} and \ref{ass:g_eps}, we obtain for $\vareps \to 0$ almost everywhere in $(0,T)$
\begin{align*}
-\int_{\Sigma} &p_0 \left[|Z_f| \nabla_{\x} \cdot U - d_n^f \Delta_{\x} V\right]  d\x
\\
&+ \int_{\Sigma}\int_{Z_s} A\left[ D_{\x}(\widetilde{u}_1) -y_n \nabla_{\x}^2 u_0^n + D_y(u_2)\right] : \left[D_{\x}(U) - y_n \nabla_{\x}^2 V\right]dy   d\x 
\\
=& \int_{\Sigma} \overline{f_1^n} V + f_0 \cdot \left[|Z_f| U - d_n^f \nabla_{\x} V\right] d\x +\int_{\Sigma} \overline{g_1^n} V + g_0 \cdot \left[|Z_s| U - d_n^s \nabla_{\x} V\right] d\x,
\end{align*}
where we used the notations (for $\ast \in \{s,f\}$)
\begin{align*}
d_n^{\ast} = \int_{Z_{\ast}} y_n dy, \quad \overline{f_1^n} := \int_{Z_f} f_1^n dy , \quad \overline{g_1^n} := \int_{Z_s} g_1^n dy.
\end{align*}
Using the representation from Proposition \ref{prop:corrector_u2} for $u_2$ we obtain after an elemental calculation (remember $\widetilde{u}_1^n = 0$)
\begin{align*}
\int_{\Sigma}&\int_{Z_s} A\left[ D_{\x}(\widetilde{u}_1) -y_n \nabla_{\x}^2 u_0^n + D_y(u_2)\right] : \left[D_{\x}(U) - y_n \nabla_{\x}^2 V\right]dy  d\x
\\
=& \int_{\Sigma} a^{\ast} D_{\x}(\widetilde{u}_1) : D_{\x} (U) + b^{\ast} \nabla_{\x}^2 u_0^n : D_{\x} (U) + b^{\ast} D_{\x}(\widetilde{u}_1) : \nabla_{\x}^2 V + c^{\ast} \nabla_{\x}^2 u_0^n : \nabla_{\x}^2 V d\x
\\
&+ \int_{\Sigma} p_0 B^1 : D_{\x}(U) + p_0 B^2 : \nabla_{\x}^2 V d\x.
\end{align*}
Altogether, we obtain the weak variation equation $\eqref{eq:weak_form_plate}$.
Hence, together with the zero boundary conditions of $\widetilde{u}_1$, $u_0^n$, and $\nabla_{\x} u_0^n$ on $\partial \Sigma$ and the zero initial conditions for $\widetilde{u}_1$ and $u_0^n$, we have shown that $(\widetilde{u}_1,u_0^n)$ is the weak solution of the plate equation $\eqref{eq:Macro_Model_elastic}$ - $\eqref{Macro_Model:plate_4th_order}$.

\subsection{Uniqueness for the macroscopic model}

To show uniqueness for the macroscopic model $\eqref{Macro_Model}$ it is enough to show that for zero data $f_0 = g_0 =0$ and $f_1^n = g_1^n =0$ we have $(p_0,\widetilde{u}_1, u_0^n) = 0$. For this we test the weak equations $\eqref{eq:weak_form_Darcy}$ and $\eqref{eq:weak_form_plate}$ with $p_0$ resp. $(\partial_t \widetilde{u}_1,\partial_t u_0^n)$ and add up the three equations to obtain
\begin{align*}
\frac12 \frac{d}{dt} \int_{\Sigma} \alpha^h p_0^2 &d\x + \int_{\Sigma} K\nabla_{\x} p_0  \cdot \nabla_{\x} p_0 d\x + \frac12 \frac{d}{dt} \int_{\Sigma} \bigg\{ a^{\ast} D_{\x}(\widetilde{u}_1) : D_{\x} (\widetilde{u}_1) 
\\
&+ b^{\ast} \nabla_{\x}^2 u_0^n : D_{\x} (\widetilde{u}_1) + b^{\ast} D_{\x}(\widetilde{u}_1) : \nabla_{\x}^2 u_0^n + c^{\ast} \nabla_{\x}^2 u_0^n : \nabla_{\x}^2 u_0^n\bigg\} d\x  = 0.
\end{align*}
Using $p_0(0) =0$, $\widetilde{u}_1(0)=0$, and $u_0^n(0)=0$, we obtain after integrating with respect to time  for almost every $t\in (0,T)$ with the positivity of $\alpha^h$ and $K$
\begin{align*}
c_0 \|p_0(t)\|_{L^2(\Sigma)}^2 + c_0& \|\nabla_{\x}  p_0 \|^2_{L^2((0,t)\times \Sigma)} + \frac12\int_{\Sigma} \bigg\{ a^{\ast} D_{\x}(\widetilde{u}_1) : D_{\x} (\widetilde{u}_1)
\\
 &+ b^{\ast} \nabla_{\x}^2 u_0^n : D_{\x} (\widetilde{u}_1) + b^{\ast} D_{\x}(\widetilde{u}_1) : \nabla_{\x}^2 u_0^n + c^{\ast} \nabla_{\x}^2 u_0^n : \nabla_{\x}^2 u_0^n \bigg\} d\x \le  0
\end{align*}
for a constant $c_0>0$.
Since the bilinear-form given by the integral is positive (see $\eqref{ineq:positivity_abc}$) we obtain (after a possible change of $c_0$)
\begin{align*}
c_0 \left(\|p_0(t)\|_{L^2(\Sigma)}^2 + \|\nabla_{\x} p_0 \|^2_{L^2((0,t)\times \Sigma)}  +  \|D_{\x}(\widetilde{u}_1)(t)\|_{L^2(\Sigma)}^2 + \|\nabla_{\x}^2 u_0(t)\|_{L^2(\Sigma)}^2 \right) \le 0.
\end{align*}
Hence, we immediately obtain $p_0 = 0$, and the zero-boundary condition on $\partial\Sigma$ for $u_0^n$, $\nabla_{\x}u_0^n$, and $\widetilde{u}_1$ implies $u_0^n=0$ and $\widetilde{u}_1=0$. This gives the uniqueness of the problem. In particular, all the convergence results obtained above are also valid for the whole sequence.

%
%
%

\appendix

\section{Properties of the space $\spaceV$}
\label{sec:two_pressure_decomp}
In this section we investigate the spaces $\spaceV $ and $\spaceV_0$ (Definitions see below), which are crucial for the derivation of the two-pressure Stokes model in Proposition \ref{prop:cell_equation_w1_Darcy_vel}. For the latter we need  two crucial properties: The two-pressure decomposition of $\spaceV^{\perp}_0$, see Lemma \ref{lem:two_pressure_decom}, and the density of smooth functions in $\spaceV$, see Lemma \ref{lem:density_V}. We emphasize that we formulate all the results in this section for time-independent functions. Since here time only acts as a paramter, the results are also valid for the time-dependent case.
\\
Let us define the spaces $\spaceV$ and $\spaceV_0$:
For a function $u \in L^2(\Sigma \times Z_f)$ we set 
\begin{align}\label{def:mean_Z_f}
\u := \int_{Z_f} u(\cdot_{\x} ,y) dy \in L^2(\Sigma).
\end{align}
We use the same definition for functions depending on time. Now, we define the space
\begin{align}
\begin{aligned}\label{def:spaceV}
\spaceV := \big\{ u \in L^2(\Sigma,& H^1_{\#}(Z_f,\Gamma))^n \, : 
\\
&\, \nabla_y \cdot u = 0 \, \mbox{ in } \Sigma \times Z_f,\, \, \nabla_{\x} \cdot \u \in L^2(\Sigma), \, \, \u \cdot \nu = 0 \mbox{ on } \partial_D \Sigma\big\}.
\end{aligned}
\end{align}
Here the normal trace on a part of the boundary of $\partial \Sigma$ has to be understood in the following sense:
For a function $w \in L^2(\Sigma)^{n-1}$ with $\nabla_{\x} \cdot w  \in L^2(\Sigma)$ we say $w \cdot \nu = 0$ on $\partial_D \Sigma$ if for all $\phi \in H^1(\Sigma,\partial_N \Sigma)$ it holds that
\begin{align}\label{def:normal_trace_zero_part_boundary}
\langle w \cdot \nu , \phi \rangle_{H^{-\frac12}(\partial\Sigma),H^{\frac12}(\partial\Sigma) } = 0.
\end{align}
On $\spaceV$ we consider the inner product
\begin{align*}
(u,v)_{\spaceV}:= (\nabla_y u,\nabla_y v)_{L^2(\Sigma\times Z_f)} + (\nabla_{\x} \cdot \u, \nabla_{\x} \cdot \v)_{L^2(\Sigma)}
\end{align*}
and denote the associated norm by $\|\cdot \|_{\spaceV}$. 
We define 
\begin{align*}
\spaceV_0:= \left\{u \in \spaceV \, : \, \nabla_{\x} \cdot \u = 0\right\}.
\end{align*}
For an arbitrary Banach space $X$ we define its annihilator by
\begin{align*}
X^{\perp}:= \left\{ F \in X' \, : \, F(x) = 0 \mbox{ for all } x \in X\right\}.
\end{align*}
Next we characterize the annihilator of $\spaceV_0$. Similar results for the pure Stokes problem in perforated domains (not thin) and a no-slip boundary condition on the interior oscillating surface can be found in \cite{allaire1997one} and \cite[Section 14]{chechkin2007homogenization}. A detailed proof for thin perforated domains and  zero normal flux boundary condition on $\Gamma$ for the fluid velocity can be found in \cite[Lemma 5.3 and Remark 5.5]{fabricius2023homogenization}. Here we follow the proof from \cite{fabricius2023homogenization}, where the only difference is the boundary condition on $\partial \Sigma$, hence we only give the basic ideas and point out the main difference.

Compared to the proof in \cite{fabricius2023homogenization}, here we make use of the following Helmholtz-decomposition, which is for sure known in the literature, but the author was not able to find a good reference, so for the sake of completeness we give a proof. 
\begin{lemma}\label{lem:Helmholtz_decomposition}
Let $\Omega \subset \R^n$ open and bounded with Lipschitz-boundary. Further, we assume for $\partial_N \Omega, \partial_D \Omega \subset \partial\Omega$ that $\partial \Omega = \overline{\partial_N \Omega} \cup \overline{\partial_D \Omega}$ and $\partial_N \Omega \cap \partial_D \Omega = \emptyset$, and that $\partial_N \Omega$ is open. Then it holds the following Helmholtz-decomposition:
\begin{align*}
L^2(\Omega)^n = L_{\sigma,D}^2(\Omega) \oplus \nabla H^1(\Omega,\partial_N \Omega )
\end{align*}
with $L_{\sigma,D}^2(\Omega) := \left\{v \in L^2(\Omega)^n \, : \, \nabla \cdot v = 0 , \, v\cdot \nu = 0 \mbox{ on } \partial_D \Omega\right\}.$
\end{lemma}
\begin{proof}
We only give the proof for $\partial_N \Omega \neq \emptyset$ (this case can be found for example in \cite[Theorem III.1.1]{Galdi}). We have 
\begin{align*}
L^2(\Omega)^n = L_{\sigma,D}^2(\Omega)^{\perp} \oplus L_{\sigma,D}^2(\Omega).
\end{align*}
Hence, we have to show $L_{\sigma,D}^2(\Omega)^{\perp}  = \nabla H^1(\Omega,\partial_N \Omega)$. The inclusion $\supset$ follows directly from integration by parts. Let us assume $w \in L_{\sigma,D}^2(\Omega)^{\perp}$, this means
\begin{align}\label{eq:Helmholtz_aux}
(w,v)_{L^2(\Omega)} = 0 \quad\mbox{ for all } v \in L_{\sigma,D}^2(\Omega).
\end{align}
Consider the divergence operator $\div: H^1(\Omega,\partial_D \Omega)^n \rightarrow L^2(\Omega)$, which is surjective (here we use that $\partial_N \Omega$ is open) and therefore closed. Hence, from the closed range theorem we get for the adjoined $\div^{\ast}: L^2(\Omega) \rightarrow \left(H^1(\Omega,\partial_D \Omega)^n\right)^{\prime}$ that $N(\div)^{\perp} = R(\div^{\ast}).$
Since $N(\div) \subset L_{\sigma,D}^2(\Omega)$ (obviously the normal trace condition is fulfilled), we obtain from $\eqref{eq:Helmholtz_aux}$ that $w \in R(\div^{\ast})$. Hence, there exists $p \in L^2(\Omega)$ such that $w = -\div^{\ast} p$ and for all $v \in H^1(\Omega,\partial_D \Omega)^n$ it holds that
\begin{align*}
(w,v)_{L^2(\Omega)}  = - \langle \div^{\ast} p , v\rangle_{H^1(\Omega,\partial_D \Omega)',H^1(\Omega,\partial_D \Omega)} = - (p,\nabla \cdot v)_{L^2(\Omega)}.
\end{align*}
This implies by definition of the weak derivative $w = \nabla p$  (choosing $v $ smooth with compact support) and by integration by parts 
\begin{align*}
(p,v \cdot \nu)_{L^2(\partial\Omega)} = 0
\end{align*}
for all $v \in H^1(\Omega,\partial_D \Omega)^n$. Since $\partial_N \Omega$ is open in $\partial \Omega$ we get $p=0$ on $\partial_N \Omega$.
\end{proof}
Now, we are able to give the two-pressure decomposition of $\spaceV_0^{\perp}$:
\begin{lemma}\label{lem:two_pressure_decom}
It holds that 
\begin{align*}
\spaceV_0^{\perp} = \left\{ \nabla_{\x} p_0 + \nabla_y p_1  \in L^2(\Sigma, H_{\#}^1(Z_f,\Gamma)^n)' \, : \, p_0 \in H^1(\Sigma,\partial_N \Sigma),\, p_1 \in L^2(\Sigma \times Z_f)\right\}.
\end{align*}
The decomposition on the right-hand side is unique.
\end{lemma}
\begin{proof}
As mentioned above we only give a sketch of the proof and refer to \cite[Lemma 5.3]{fabricius2023homogenization} for the details. We define
\begin{align*}
\spaceV_1:=& \left\{\phi \in L^2(\Sigma,H_{\#}^1(Z_f,\Gamma))^n\, : \, \nabla_y \cdot \phi = 0 \mbox{ in } \Sigma \times Z_f\right\},
\\
\spaceV_2:=& \left\{\phi \in L^2(\Sigma,H_{\#}^1(Z_f,\Gamma))^n\, : \, \nabla_{\x} \cdot \overline{\phi} = 0,\, \overline{\phi}\cdot \nu = 0 \mbox{ on } \partial_D \Sigma\right\}.
\end{align*}
We have $\spaceV_0 = \spaceV_1 \cap \spaceV_2$ and therefore $\spaceV_0^{\perp} = \spaceV_1^{\perp} \oplus \spaceV_2^{\perp}$. The surjectivity of the divergence operator $\div: H_{\#}^1(Z_f,\Gamma)^n\rightarrow L^2(Z_f)/\R$ and the closed range theorem imply
\begin{align*}
\spaceV_1^{\perp} = \left\{ \nabla_y p_1 \, : \, p_1 \in L^2(\Sigma, L^2(Z_f)/\R)\right\}.
\end{align*}
It remains to show  $\spaceV_2^{\perp} = \left\{\nabla_{\x} p_0 \, : \, p_0 \in H^1(\Sigma,\partial_N \Sigma)\right\}.$
Let $h\in H_{\#}^1(Z_f,\Gamma)$ solve  $-\Delta_y h = 1$ in $Z_f$.
In particular, we have $\bar{h}:= \int_{Z_f} h dy >0$. We define the mapping $F:L^2(\Sigma)^{n-1} \rightarrow L^2(\Sigma,H^1_{\#}(Z_f,\Gamma))^n$ by $F(\theta):= \phi_{\theta}$ ($n$-th component equal zero) with 
\begin{align*}
\phi_{\theta}(\x,y):= \frac{1}{\bar{h}} \theta(\x) h(y).
\end{align*}
It is easy to check that $F$ is linear and bounded with closed range $R(F)$. Further, we  have (this follows from the equation for $h$)
\begin{align*}
R(F)^{\perp} = \left\{\psi \in L^2(\Sigma,H_{\#}^1(Z_f,\Gamma))^n \, : \, \int_{Z_f} \psi_i dy = 0 \mbox{ for } i=1,\ldots,n-1\right\}.
\end{align*}
Now, for $G \in \spaceV_2^{\perp} $ the operator $G\circ F$ is a bounded and linear operator from $L^2(\Sigma)^{n-1}$ to $\R$ and we can identify it with $g \in L^2(\Sigma)^{n-1}$. Using the Helmholtz-decomposition from Lemma \ref{lem:Helmholtz_decomposition} we obtain $g = g_0 + \nabla_{\x} p_0$ with $g_0 \in L_{\sigma,D}^2(\Sigma)$ and $p_0 \in H^1(\Omega,\partial_N \Sigma)$. Since for every $\theta \in L^2_{\sigma,D}(\Sigma)$ we have $F(\theta) = \phi_{\theta} \in \spaceV_2$ we have $G\circ F(\theta) = 0$ and therefore $g_0 = 0$. Hence, we have $G=\nabla_{\x} p_0$ on $R(F)$. Since both operators $G$ and $\nabla_{\x} p_0$ vanish on $R(F)^{\perp}$, we obtain $G= \nabla_{\x} p_0$ on $L^2(\Sigma,H_{\#}^1(Z_f,\Gamma))^n$ which finishes the proof.
\end{proof}

\subsection{A density result}

In the two-pressure Stokes model, or more precisely in its derivation, we use test-functions from the space $\spaceV$. However, for $\phi \in \spaceV$ the function $\phi\left(\cdot_{\x},\frac{\cdot_x}{\vareps}\right)$ is in general not measurable and therefore not an admissible test-function for the microscopic equation $\eqref{eq:micro_model}$. Hence, we approximate such functions by more regular functions. We introduce the subspace $\spaceV^{\infty}$ of $\spaceV$ defined by 
\begin{align*}
\spaceV^{\infty}:= \left\{ u \in C^{\infty}_0( \overline{\Sigma}\setminus \partial_D \Sigma , H^1_{\#}(Z_f,\Gamma))^n\, : \, \nabla_y \cdot u \mbox{ in } \Sigma \times Z_f\right\}.
\end{align*}
We have the following density result which is a special case of the more general density result we give below.
\begin{lemma}\label{lem:density_V}
The space $\spaceV^{\infty}$ is dense in $\spaceV$.
\end{lemma}
\begin{proof}
See Lemma \ref{lem:Density_general} below.
\end{proof}

In the following we denote for an open and bounded subset $U\subset \R^n$ by $X(U)$ a Banach space of functions $u:U\rightarrow \R$, such that $X(U) \subset L^1(U)$.  In particular, the mean value is well-defined. For a function $u:\Sigma \rightarrow X(U)$ we denote its mean value by
\begin{align*}
\u(x):= \int_U u(x,y) dy.
\end{align*}
We show the density of $C_0^{\infty}\left( \overline{\Sigma} \setminus \partial_D \Sigma, X(U)\right)$ in 
\begin{align*}
\spaceW := \left\{ u \in L^2(\Sigma,X(U))\, : \, \nabla_{\x} \cdot \u \in L^2(\Sigma) , \, \u \cdot \nu = 0 \, \mbox{ on } \partial_D \Sigma \right\},
\end{align*}
with respect to the norm
\begin{align*}
\|u\|_{\spaceW}^2 := \|u\|_{L^2(\Sigma,X(U))}^2 + \|\nabla_{\x} \cdot \u \|_{L^2(\Sigma)}^2. 
\end{align*}
As an immediate consequence we obtain the density of $\spaceV^{\infty}$ in $\spaceV$.
For the proof we use the convolution which commutes with the mean value operator. We follow the usual ideas for the density of smooth functions with suitable boundary conditions in the $H(\div)$-spaces (also with suitable boundary conditions), see for example \cite{Galdi}. In our case, a crucial point is the mixed boundary condition on $\partial \Sigma$ for the space $\spaceW$. On $\partial_D \Sigma$ we have a normal-trace zero boundary condition and on $\partial_N \Sigma$ we have arbitrary boundary conditions. For similar results in the scalar valued case $X(U) = \R$ and curved domains we refer to \cite{bauer2016maxwell}.

\begin{lemma}\label{lem:Density_general}
$C_0^{\infty}\left( \overline{\Sigma} \setminus \Gamma_D, X(U)\right)$ is dense in $\spaceW$.
\end{lemma}
\begin{proof}
For the sake of simplicity we consider only the case $\Sigma:= (-1,1)^2 $ with
\begin{align*}
\partial_D \Sigma := \{\pm 1\} \times (-1,1),\quad \partial_N \Sigma := (-1,1) \times \{\pm 1\}.
\end{align*}
In the following we use the short notation 
\begin{align*}
\langle \cdot , \cdot \rangle := \langle \cdot ,\cdot \rangle_{H^{-\frac12}(\partial \Sigma),H^{\frac12}(\partial\Sigma)}.
\end{align*}
Let $u\in \spaceW$. We extend the function $u$ by zero to $\Sigma_{\infty}:= \R\times (-1,1)$ and denote this function again by $u$. We have
\begin{align*}
\u\in H(\div,\Sigma_{\infty}):= \left\{ v \in L^2(\Sigma)^2 \, : \,  \nabla_{\x} \cdot v \in L^2(\Sigma)\right\}.
\end{align*}
In fact, let $\phi \in C_0^{\infty}(\Sigma_{\infty})$. Then we have
\begin{align*}
\int_{\Sigma_{\infty}} \nabla_{\x} \phi \cdot \u d\x &= \int_{\Sigma_{\infty} \setminus \Sigma} \nabla_{\x} \phi \cdot \underbrace{\u}_{=0} d\x + \int_{\Sigma} \nabla_{\x} \phi \cdot \u d\x
\\
&= -\int_{\Sigma} \phi \nabla_{\x} \cdot \u d\x +  \langle \u \cdot \nu , \phi \rangle
= - \int_{\Sigma_{\infty}} \phi \widetilde{\nabla_{\x} \cdot \u} d\x
\end{align*}
with the zero extension $\widetilde{\nabla_{\x} \cdot \u} \in L^2(\Sigma_{\infty})$ of $\nabla_{\x} \cdot \u$ to $\Sigma_{\infty}$. We emphasize that the boundary term in the calculation above vanishes, since $\phi = 0$ on $\partial_N \Sigma$. This implies $\u \in H(\div,\Sigma_{\infty})$. 
Next, we choose $\rho:= (\rho_1,\rho_2) $ with $\rho_1 > 1$ and $0< \rho_2 < 1$. Later we consider $\rho \rightarrow (1,1)$. We define
\begin{align*}
\Sigma_{\infty}^{\rho_2} := \R \times \left(-\frac{1}{\rho_2},\frac{1}{\rho_2}\right) \supset \Sigma_{\infty} \supset \Sigma,
\end{align*}
and $u^{\rho} : \Sigma_{\infty}^{\rho_2} \times U \rightarrow \R^2$ for almost every $(\x,y) \in \Sigma_{\infty}^{\rho_2} \times U$ by (similar to Piola transformation)
\begin{align*}
u^{\rho}(\x,y) := \left(
\rho_1^{-1} u_1 (\rho_1 x_1, \rho_2 x_2,y) ,
\rho_2^{-1} u_2 (\rho_1 x_1 , \rho_2 x_2,y)  \right)^t
\end{align*}
Obviously, we have
\begin{align}\label{eq:identity_u_rho}
\overline{u^{\rho}}(\x) = \u^{\rho}(\x) := \left( \rho_1^{-1}  \u_1 (\rho_1 x_1, \rho_2 x_2)  , \rho_2^{-1} \u_2 (\rho_1 x_1 , \rho_2 x_2) \right)^t.
\end{align}
Since the Piola transformation preserves the weak divergence, we get $\u^{\rho} \in H(\div,\Sigma_{\infty}^{\rho_2})$ with (for almost every $\x \in \Sigma_{\infty}^{\rho_2}$)
\begin{align}\label{eq:identity_div_u_rho}
\nabla_{\x} \cdot \u^{\rho}(\x) = (\nabla_{\x} \cdot \u)^{\rho}(\x) := \nabla_{\x} \cdot \u \left(\rho_1 x_1 , \rho_2 x_2\right).
\end{align}
Now, for $0 < \vareps \ll 1$ and $v \in L^2(\R^2,Z)^m$ with $m \in \N$ and a Banach space $Z$,
we consider the convolution (with respect to $\x$) for $\x \in \R^2$ 
\begin{align*}
v (\x):= \int_{\R^2} \phi_{\vareps}(\x - \z ) v(\z) d\z
\end{align*}
with a smooth convolution kernel $\phi_{\vareps} \in C_0^{\infty}(\R^2)$ with $\phieps \geq 0$, $\mathrm{supp}(\phieps) \subset B_{\vareps}(0)$ and $\int_{\R^2} \phieps d\x = 1$. In the following, if not stated otherwise, we extend all functions by zero to $\R^2$.

For $\vareps < \rho_1 - 1$ we have $\left(u^{\rho}\right)_{\varepsilon}  \in C_0^{\infty}(\overline{\Sigma}\setminus \partial_D \Sigma, X(U))$, since $\rho_1>1$ and $u$ is zero in $(\R\times (-1,1)) \setminus \overline{\Sigma}$.
Now, we have
\begin{align*}
\left\| u -\left(u^{\rho}\right)_{\varepsilon} \right\|_{L^2(\Sigma,X(U))} &\le \|u - u^{\rho} \|_{L^2(\Sigma,X(U))} + \left\|u^{\rho} - \left(u^{\rho}\right)_{\vareps} \right\|_{L^2(\Sigma,X(U))}
\end{align*}
By approximating $u$ with continuous functions the first term tends to zero for $\rho \to (1,1)$. The second term tends to zero (for fixed $\rho$) for $\vareps \to 0$, due to the properties of the convolution. It remains to estimate the norm of the divergence of the mean value: First of all, we notice that the convolution commutes with the mean value, more precisely, we have $ \overline{\left(u^{\rho}\right)_{\vareps}} = \left( \u^{\rho}\right)_{\vareps}$ (see also $\eqref{eq:identity_u_rho}$)
Hence, we obtain in $\Sigma$ (we emphasize that $\rho_2>1$)
\begin{align*}
\nabla_{\x} \cdot \overline{\left(u^{\rho}\right)_{\vareps}} = \left(\nabla_{\x} \cdot \u^{\rho} \right)_{\vareps}.
\end{align*}
With this we get
\begin{align*}
\left\|\nabla_{\x} \cdot \u - \nabla_{\x} \cdot \overline{\left(u^{\rho}\right)_{\vareps}} \right\|_{L^2(\Sigma)} \le \|\nabla_{\x} \cdot \u - \nabla_{\x} \cdot \u^{\rho}\|_{L^2(\Sigma)} + \left\|\nabla_{\x} \cdot \u^{\rho} - \left(\nabla_{\x} \cdot \u^{\rho}\right)_{\vareps} \right\|_{L^2(\Sigma)}.
\end{align*}
Using $\eqref{eq:identity_div_u_rho}$ and similar arguments as above (approximating $\nabla_{\x} \cdot \u$ with continuous functions), the first term vanishes for $\rho \to (1,1)$, and the second term for $\vareps \to 0$. This gives the desired result.
\end{proof}

\section{Bogovskii-operator for thin perforated domains}
\label{sec:Bogovskii}
Here we define the Bogovskii-operator for thin perforated domains and give $\vareps$-uniform bounds for its norm. For this we make use of the restriction operator, see \cite[Appendix]{SanchezPalencia1980} and \cite{allaire1989homogenization} for perforated domains, and for thin perforated domains in \cite{fabricius2023homogenization}.

\begin{lemma}\label{lem:Restriction_Operator}
There exists a linear operator
\[
R_\vareps\colon H^1(\oeps,S_\vareps \cup \partial_D \oeps)^n \rightarrow H^1(\oef, \geps \cup \partial_D \oef)^n
\]
such that for all $\veps \in H^1(\oeps,S_\vareps \cup \partial_D \oeps)^n$ we have:
\begin{enumerate}[label = (\alph*)]
 \item For $k\in K_{\vareps}$ it holds almost everyhwhere in $\vareps (Z + k)$
 \begin{align}\label{eq:restriction_operator}
 \nabla \cdot R_{\vareps} \veps = \nabla \cdot \veps + \frac{1}{|\vareps Z_f|} \int_{\vareps (Z_s + k)} \nabla \cdot \veps dz.
 \end{align}
%
    \item It holds that
    \begin{align*}
        \Vert R_\vareps \veps \Vert_{L^2(\oef)} + \varepsilon \Vert \nabla R_\vareps \veps \Vert_{L^2(\oef)} \le C\vareps \Vert \nabla \veps \Vert_{L^2(\oeps)}.
    \end{align*}
\end{enumerate}

\end{lemma}
\begin{proof}
This result was shown for slightly different boundary conditions in \cite[Proposition 4.5]{fabricius2023homogenization}. More precisely, it was shown the existence of $R_{\vareps}: H^1(\oeps,S_{\vareps})^n\rightarrow H^1(\oef, \geps)^n$ with the desired property, where for $\eqref{eq:restriction_operator}$ we refer to the property (iii)  of the local restriction operator in the proof of \cite[Proposition 4.5]{fabricius2023homogenization}. From property (iii) for the local restriction operator in the proof, we get that $R_{\vareps}$ restricted to $H^1(\oeps,S_{\vareps} \cup  \partial_D \oeps)^n$ maps to $H^1(\oef, \geps \cup \partial_D \oef)^n$, which gives the desired result.
\end{proof}
As a direct consequence we obtain the following Bogovskii-operator. A similar result was shown in \cite[Theorem 2.2]{fabricius2022pressure} where the interface $\geps$ was given as a graph. Here we show the general result.
\begin{corollary}\label{cor:Bogovskii}
For every $f_{\vareps} \in L^2(\oef)$ there exists $\phieps \in H^1(\oef, \geps \cup \partial_D \oef)^n$ such that
\begin{align*}
\nabla \cdot \phieps = f_{\vareps}, \qquad \|\phieps\|_{L^2(\oef)} + \vareps \|\nabla \phieps\|_{L^2(\oef)} \le C \|f_{\vareps}\|_{L^2(\oef)}
\end{align*}
for a constant $C>0$ independent of $\vareps$. In other words, there exists a linear operator $B_{\vareps}: L^2(\oef) \rightarrow H^1(\oef,\geps \cup \partial_D \oef)^n$ such that $\nabla \cdot B_{\vareps} f_{\vareps} = f_{\vareps}$ and the operator norm of $B_{\vareps}$ fulfills $\|B_{\vareps}\|\le C \vareps^{-1}$.
\end{corollary}
\begin{proof}
Let $f_{\vareps} \in L^2(\oef)$ and denote its extension by zero to the whole layer $\oeps$ by $\widetilde{f}_{\vareps}$. By a simple scaling argument (see also \cite{fabricius2022pressure} for more details) and the well known Bogovskii-result, we obtain the existence of $\tphieps \in H^1(\oeps,\partial_D \oeps \cup S_{\vareps})^n$ such that
\begin{align*}
\nabla \cdot \tphieps = \widetilde{f}_{\vareps}, \qquad \| \tphieps\|_{L^2(\oeps)} + \vareps \|\nabla \tphieps \|_{L^2(\oeps)} \le C \|f_{\vareps}\|_{L^2(\oef)}.
\end{align*}
With $\phieps:= R_{\vareps} \tphieps$ and Lemma \ref{lem:Restriction_Operator} we obtain the desired result.
\end{proof}

\section{Important inequalities}
\label{sec:inequalities}
Here we summarize some important inequalities from the literature related to the treatment of homogenization problems in thin perforated domains. We start with some standard Poincar\'e and Korn inequalities for perforated domains with vanishing traces on the oscillating boundary. They follow by a simple decomposition argument.

\begin{lemma}[Poincar\'e and Korn inequaly]\label{lem:ineq_Poincare_Korn}
Let $\ast \in \{f,s\}$. 
\begin{enumerate}
[label = (\roman*)]
\item For all $\phieps \in H^1(\oeps^{\ast},\geps)$ it holds that
\begin{align*}
\|\phieps\|_{L^2(\oeps^{\ast})} \le C \vareps \|\nabla \phieps\|_{L^2(\oeps^{\ast})}.
\end{align*}

\item  For all $\weps \in H^1(\oeps^{\ast},\geps)^n$ it holds that
\begin{align*}
\|\weps\|_{L^2(\oeps^{\ast})} + \vareps \|\nabla \weps\|_{L^2(\oeps^{\ast})} \le C\vareps \|D(\weps)\|_{L^2(\oeps^{\ast})}.
\end{align*}
\end{enumerate}

\end{lemma}

The next Korn-inequality is valid for thin perforated domains when the traces only vanish on the lateral (perhaps oscillating) part of the boundary.
\begin{lemma}\label{lem:app_Korn_thin}
For all $\weps \in H^1(\oeps^{\ast},\partial_D \oeps^{\ast})^n$ with $\ast \in\{ f,s\}$ it holds that 
\begin{align*}
\sum_{i=1}^{n-1}  \foe \Vert \weps^i &\Vert_{L^2(\oeps^{\ast})} + \sum_{i,j=1}^{n-1} \foe\Vert \partial_i \weps^j \Vert_{L^2(\oeps^{\ast})} 
\\
&+ \Vert \weps^n\Vert_{L^2(\oeps^{\ast})} +  \Vert \nabla \weps \Vert_{L^2(\oeps^{\ast})} \le \frac{C}{\vareps} \Vert D(\weps)\Vert_{L^2(\oeps^{\ast})}.
\end{align*}	
\end{lemma}
\begin{proof}
See \cite[Theorem 2]{GahnJaegerTwoScaleTools}.
\end{proof}
For a function $\weps \in H^1(\oeps)^n$ with vanishing traces on the lateral boundary part $\partial_D \oes$ we can control the rest of the norm on the lateral boundary in the following way:
\begin{lemma}\label{lem:app_trace_lateral_BC}
For every $\weps \in H^1(\oeps)^n$ with $\weps = 0 $ on $\partial_D \oes$ it holds that
\begin{align*}
\|\weps\|_{L^2(\partial \Sigma \times (-\vareps , \vareps))} \le C \sqrt{\vareps} \|D(\weps)\|_{L^2(\oes)}
\end{align*}
for a constant $C>0$ independent of $\vareps$.
\end{lemma}
\begin{proof}
See \cite[Lemma 4]{GahnJaegerTwoScaleTools}.
\end{proof}
Finally, we formulate an extension result for (thin) perforated domains which allows to control the norm of the gradient resp. the symmetric gradient of the extended function by the norm of the gradient resp. the symmetric gradient uniformly with respect to the scaling parameter $\vareps$. This allows to treat functions on perforated domains in a similar way as functions defined in the whole thin layer.

\begin{lemma}\label{lem:Extension_operator}
Let $\ast \in \{f,s\}$. There exists an extension operator $E_{\vareps} : H^1(\oeps^{\ast})^n \rightarrow H^1(\oeps)^n$, such that for all $\veps \in H^1(\oeps^{\ast})^n$ it holds that ($i=1,\ldots,n$)
\begin{align*}
\Vert (E_{\vareps}\veps)^i \Vert_{L^2(\oeps)} &\le C \left( \Vert \veps^i \Vert_{L^2(\oeps^{\ast})} + \vareps \Vert \nabla \veps \Vert_{L^2(\oeps^{\ast})} \right),
\\
\Vert \nabla E_{\vareps}\veps \Vert_{L^2(\oeps)} &\le  C \Vert \nabla \veps \Vert_{L^2(\oeps^{\ast})},
\\
\Vert D(E_{\vareps}\veps) \Vert_{L^2(\oeps)} &\le C \Vert D(\veps)\Vert_{L^2(\oeps^{\ast})},
\end{align*}
for a constant $C > 0$ independent of $\varepsilon$. 
\end{lemma}
\begin{proof}
See \cite[Theorem 1]{GahnJaegerTwoScaleTools}.
\end{proof}

\section{Two-scale convergence in thin domains}
\label{sec:two_scale_convergence}

We briefly introduce two-scale convergence concepts for thin layers \cite{BhattacharyaGahnNeussRadu,MarusicMarusicPalokaTwoScaleConvergenceThinDomains,NeussJaeger_EffectiveTransmission} (for domains we refer to the seminal works \cite{Allaire_TwoScaleKonvergenz,Nguetseng}), and recall the compactness results used in this paper. Here, we use $\ast \in \{f,s\} $ and therefore $\oeps^{\ast}$ denotes the domain $\oef$ or $\oes$.
\begin{definition}\
\begin{enumerate}
[label = (\roman*)]
\item  We say the sequence $\weps \in L^2( (0,T) \times \oeps)$ converges (weakly) in the two-scale sense to a limit function $w_0 \in L^2((0,T) \times \Sigma \times Z)$ if 
\begin{align*}
\lim_{\varepsilon\to 0} \foe  \int_0^T \int_{\oeps} \weps(x) \phi \bxfxe dx dt= \int_0^T\int_{\Sigma} \int_Z w_0(\x,y) \phi(\x,y) dy d\x dt
\end{align*}
for all $\phi \in L^2((0,T) \times \Sigma,C_{\#}^0(\overline{Z}))$. We write 
\begin{align*}
\weps \rats w_0.
\end{align*}
\item We say the sequence $\weps \in L^2( (0,T) \times \geps)$ converges (weakly) in the two-scale sense to a limit function $w_0 \in L^2((0,T) \times  \Sigma \times \Gamma)$ if 
\begin{align*}
\lim_{\varepsilon\to 0}  \int_0^T  \int_{\geps} \weps(x) \phi \bxfxe d\sigma  dt = \int_0^T \int_{\Sigma} \int_{\Gamma} w_0(\x,y) \phi(\x,y) d\sigma_y d\x  dt
\end{align*}
for all $\phi \in C^0([0,T] \times \overline{\Sigma},C_{\#}^0(\Gamma))$.  We write  
\begin{align*}
\weps \rats w_0 \qquad \mbox{on } \geps.
\end{align*}
\end{enumerate}
\end{definition}
The following lemma gives basic compactness results for the two-scale convergence in thin layers. 
\begin{lemma}\label{lem:app_ts_comp_general}\
\begin{enumerate}
[label = (\roman*)]
\item Let $\weps \in L^2((0,T)\times \oeps)$ with 
\begin{align*}
\|\weps \|_{L^2((0,T)\times \oeps)} \le C\sqrt{\vareps}.
\end{align*}
Then there exists a subsequence (again denoted $\weps$) and a limit function $w_0 \in L^2((0,T)\times \Sigma \times Z)$ such that
\begin{align*}
\weps \rats w_0.
\end{align*}
\item Let $\weps \in L^2((0,T) , H^1(\oeps))$ be a sequence with
\begin{align*}
\frac{1}{\sqrt{\varepsilon}}\Vert \weps \Vert_{L^2((0,T) \times  \oeps)} + \sqrt{\varepsilon}\Vert \nabla \weps \Vert_{L^2( (0,T) \times \oeps)}  \le C.
\end{align*}
Then there exists a subsequence (again denoted by $\weps$) and a limit function $w_0 \in  L^2( (0,T) \times  \Sigma, H_{\#}^1(Z))$ such that
\begin{align*}
\weps  \rats w_0, \qquad
\varepsilon \nabla \weps \rats \nabla_y w_0.
\end{align*}
\item\label{comp:ts_conv}  Consider the sequence $\weps \in L^2((0,T) \times  \geps)$ with
\begin{align*}
\Vert \weps \Vert_{L^2((0,T) \times \geps)} \le C.
\end{align*}
Then there exists a subsequence (again denoted by $\weps$) and a limit function $w_0 \in  L^2( (0,T) \times  \Sigma\times \Gamma)$ such that  
\begin{align*}
\weps \rats w_0 \qquad \mbox{on } \geps.
\end{align*}
\end{enumerate}
\end{lemma}
\begin{proof}
 For \ref{comp:ts_conv} see \cite{NeussJaeger_EffectiveTransmission}. The second result can be found in \cite{BhattacharyaGahnNeussRadu} for surfaces of thin channels and the same arguments hold for our geometrical setting.
\end{proof}
The following lemma is a consequence of the previous lemma and the existence of good extension operators, see \cite{Acerbi1992}.
\begin{lemma}\label{TwoScaleCompactnessPerforated}
Let $\weps \in L^2((0,T),H^1(\oeps^{\ast}))$ be a sequence with 
\begin{align*}
    \frac{1}{\sqrt{\varepsilon}}\Vert \weps \Vert_{L^2( (0,T)\times \oeps^{\ast})} + \sqrt{\varepsilon}\Vert \nabla \weps \Vert_{L^2((0,T)\times  \oeps^{\ast})}  \le C.
\end{align*}
Then, there exists $w_0 \in L^2((0,T)\times \Sigma, H_{\#}^1(Z_{\ast}))$ and a subsequence such that 
\begin{align*}
   \chi_{\oeps^{\ast}} \weps   \rats \chi_{Z_{\ast}} w_0,\quad
\chi_{\oeps^{\ast}}\varepsilon \nabla \weps \rats \chi_{Z_{\ast}} \nabla_y w_0,\quad
\weps\vert_{\geps} \rats w_0\vert_{\Gamma_D}.
\end{align*}
\end{lemma}

\begin{lemma}\label{lem:app_TS_Plate}
Let $\weps \in L^2((0,T),H^1(\oeps,\partial_D \oeps^{\ast}))^n$ be a sequence with
\begin{align}\label{ineq:bound_ass_TS_plate}
 \Vert D(\weps)\Vert_{L^2( (0,T)\times \oeps^{\ast})} \le C\vareps^{\frac32}.
\end{align}
Then there exist $w_0^n \in L^2((0,T),H^2_0(\Sigma))$ and $\widetilde{w}_1 \in L^2((0,T),H_0^1(\Sigma))^n$ with $\widetilde{w}_1^n = 0$ and $w_2 \in L^2((0,T)\times \Sigma , H_{\#}^1(Z)/\R))^n$, such that up to a subsequence ($\alpha = 1,\ldots,n-1$)
\begin{align*}
\weps^n &\rats w_0^n,
\\
\frac{\weps^{\alpha}}{\vareps} &\rats \big(\widetilde{w}_1^{\alpha} - y_n \partial_{\alpha} w_0^n\big),
\\
\frac{1}{\vareps} \chi_{\oeps^{\ast}} D(\weps) &\rats   \left(D_{\x}(\widetilde{w}_1) - y_n \nabla_{\x}^2 w_0^n + D_y(w_2) \right).
\end{align*}
\end{lemma}
\begin{proof}
See \cite[Theorem 3]{GahnJaegerTwoScaleTools}.
\end{proof}
We emphasize that from the bound $\eqref{ineq:bound_ass_TS_plate}$ we immediately obtain with the Korn inequality from Lemma \ref{lem:app_Korn_thin} a bound for $\weps$ and its derivatives.

\bibliographystyle{abbrv}
\bibliography{../../Referenzen/literature} 

\end{document}